\documentclass[11pt]{article}
\usepackage{amsmath}
\usepackage{amsfonts}
\usepackage{amsmath,amsthm}
\usepackage{amsmath,amssymb,amsthm,latexsym}
\usepackage{amscd}
\usepackage[all]{xy}
\usepackage{hyperref}

\newtheorem{theorem}{Theorem}[section]
\newtheorem{proposition}[theorem]{Proposition}
\newtheorem{definition}[theorem]{Definition}
\newtheorem{corollary}[theorem]{Corollary}
\newtheorem{lemma}[theorem]{Lemma}
\newtheorem{example}[theorem]{Example}
\newtheorem{problem}{Problem}
\theoremstyle{remark}
\newtheorem{remark}[theorem]{Remark}
\def\vec{\mathbf}

\begin{document}
\title{\bf{Global geometry and topology of spacelike stationary surfaces in the 4-dimensional Lorentz space}}

\author{Zhiyu Liu \and Xiang Ma%
  \thanks{LMAM, School of Mathematical Sciences, Peking University,
 Beijing 100871, People's Republic of China. \texttt{maxiang@math.pku.edu.cn}, Tel:+86-10-62767576, Fax:+86-10-62751801. Corresponding author.}\and
Changping Wang%
  \thanks{College of Mathematics and Computer Sciences,
Fujian Normal University, Fuzhou 350108, People's Republic of China. \texttt{cpwang@fjnu.edu.cn}}
\and Peng Wang%
  \thanks{Department of Mathematics, Tongji University, Shanghai 200092, People's Republic of China. \texttt{netwangpeng@tongji.edu.cn}}}

\date{\today}

\maketitle

\begin{center}
{\bf Abstract}
\end{center}

For spacelike stationary (i.e. zero mean curvature) surfaces in 4-dimensional Lorentz space, one can naturally introduce two Gauss maps and a Weierstrass-type representation. In this paper we investigate the global geometry of such surfaces systematically. The total Gaussian curvature is related with the surface topology as well as the indices of the so-called good singular ends by a Gauss-Bonnet type formula. On the other hand, as shown by a family of counter-examples to Osserman's theorem, finite total curvature no longer implies that Gauss maps extend to the ends. Interesting examples include the deformations of the classical catenoid, the helicoid, the Enneper surface, and Jorge-Meeks' k-noids. Each family of these generalizations includes embedded examples in the 4-dimensional Lorentz space, showing a sharp contrast with the 3-dimensional case.

\hspace{2mm}


{\bf MSC(2000):\hspace{2mm} 53A10, 53C42, 53C45}

\tableofcontents

\section{Introduction}

Zero mean curvature spacelike surfaces in 4-dimensional Lorentz
space $\mathbb{R}^4_1$ include classical minimal surfaces in
$\mathbb{R}^3$ and maximal surfaces in $\mathbb{R}^3_1$ as special
cases. They are no longer local minimizer or maximizer of the area
functional. Hence we call them \emph{stationary surfaces} which are
always assumed to be spacelike in this paper.

We became interested in this topic when studying several surface
classes arising from variational problems in M\"obius geometry
and Laguerre geometry \cite{Ma-Wang, CPWang2008}. After reducing our
original problems to stationary surfaces in $\mathbb{R}^4_1$, we
searched the literature and found very few papers on this general
case, which contrasted sharply to the rich theory and deep results
on minimal surfaces in $\mathbb{R}^3$ (see the recent survey
\cite{Meeks3}) and maximal surfaces in $\mathbb{R}^3_1$
\cite{Umehara}.

This situation motivated us to extend the general theory about the
global geometry and topology of minimal surfaces in $\mathbb{R}^3$ to
these stationary surfaces in $\mathbb{R}^4_1$.
As a preparation, in Section~2 we introduce the basic
invariants, equations, and the Weierstrass type representation
formula in terms of two meromorphic functions $\phi,\psi$
(corresponding to two Gauss maps, namely the two lightlike normal
directions) and a holomorphic 1-form $\mathrm{d}h$ (the height differential).
These are quite similar to the theory of minimal surfaces in $\mathbb{R}^3$.

However, stationary surfaces have quite different global geometry. According to the
classical Osserman's theorem, a complete minimal surface in
$\mathbb{R}^3$ with finite total Gaussian curvature always has
a well-defined limit of the Gauss map at each end. In contrast with
this, in Section~3 we construct complete stationary surfaces with
finite total curvature whose Gauss maps $\phi,\psi$ both have an essential
singularity at one of the ends (hence the Weierstrass data could not
extend analytically to the whole compactified Riemann surface).

Moreover, the behavior of the ends in $\mathbb{R}^4_1$ is more complicated. There might exist the
so-called \emph{singular end} where the limits of the two lightlike normal directions
coincide. In terms of the Weierstrass data $\phi,\psi$, at the end we have
$\phi=\bar\psi$. When they take this limit value with the same multiplicity,
the total Gaussian curvature will diverge. Such an end is called \emph{bad singular
end}. In the other case we define an index for every \emph{good singular end}.
These are discussed in Section~4. We then derive a Gauss-Bonnet type theorem in
Section~5 for algebraic stationary surfaces (i.e. the Weierstrass data
extend to meromorphic functions/forms on compact Riemann surfaces)
without bad singular ends. This generalizes the Jorge-Meeks formula
in $\mathbb{R}^3$ and provides a cornerstone for constructing
various examples with given global behavior.

Although there are rich examples of minimal surfaces in $\mathbb{R}^3$ locally, it is surprising that under global assumptions like completeness, embeddedness together with the topological type of the underlying surface, usually there are very few examples. This is one of the most attractive features of this research field. Here we just mention three most famous uniqueness results in $\mathbb{R}^3$ as below. \\

\noindent
{\bf Theorem} (Schoen \cite{Schoen}):~
A complete, connected minimal surfaces in $\mathbb{R}^3$ with two embedded ends and
finite total curvature is a catenoid.\\

\noindent
{\bf Theorem} (Lopez-Ros theorem \cite{Ros}):~
Complete, embedded minimal surfaces in $\mathbb{R}^3$ with genus $0$ and
finite total curvature are planes and catenoids.\\

\noindent
{\bf Theorem} (Meeks and Rosenberg \cite{Meeks2}):~
The helicoid is the only non-flat, properly embedded, simply connected,
complete minimal surfaces in $\mathbb{R}^3$.\\

In $\mathbb{R}^4_1$, is the embeddedness condition still a strong restriction on the existence of global examples? The answer is negative.

From Section~6 to Section~8 we generalize
classical minimal surfaces like the catenoid, the helicoid, the Enneper surface, and
Jorge-Meeks' $k$-noids. Each class includes embedded examples in $\mathbb{R}^4_1$.
Thus embeddedness is not so restrictive
for stationary surfaces in $\mathbb{R}^4_1$. In particular, many uniqueness theorems for completely embedded minimal surfaces
in $\mathbb{R}^3$ do not generalize to our case.
As to Schoen's theorem, this is shown by our Example~\ref{exa-graph2} which is a graph over a punctured plane; for Lopez-Ros theorem and Meeks-Rosenberg theorem, this is shown by the existence of embedded Enneper type surfaces (see Example~\ref{exa-enneper} and Proposition~\ref{prop-enneper}).

The reader might think that in such a codimension-2 case it would be easy
to deform any surface and avoid self-intersection.
But one should keep in mind that stationary surfaces are rigid objects whose global behavior is determined by a small piece.
Even though one can deform it locally, there are strong global restrictions such as being spacelike and complete.
Besides these considerations, transversal intersection is still possible
which could not be eliminated by small perturbation.
Thus the existence of many complete and embedded examples in $\mathbb{R}^4_1$ is a non-trivial fact.
It is a very interesting question whether one can establish some
uniqueness theorem under the assumption of embedding
in this new context. See discussions in Section~9 among other open problems.

It is worth pointing out that, at the beginning the classical theory of minimal surfaces motivated the generalization to other manifolds, like a series works on minimal surfaces and CMC surfaces in $S^2\times\mathbb{R}, H^2\times\mathbb{R}$,
and other homogeneous 3-manifolds \cite{Rosenberg, Meeks4, Daniel}; however, at a later stage, this development might provide new insight and tools for the original topics in $\mathbb{R}^3$. One evidence is the recent work by Hoffman, Traizet, and White \cite{Hoffman1, Hoffman2}. In these two preprints, based on the properties of helicoidal minimal surfaces in $S^2\times\mathbb{R}$, they are able to construct helicoidal minimal surfaces in $\mathbb{R}^3$ with prescribed genus. Our hope is that the study of stationary surfaces in $\mathbb{R}^4_1$ might reward us in a similar manner in the future.\\

In the final part of this introduction, we would like to mention some previous works on stationary
surfaces which are also important motivation to our work.
Estudillo and Romero \cite{Romero} considered the
exceptional value problem for the normal directions and established
certain Bernstein type theorems for complete stationary surfaces in
$\mathbb{R}^n_1$. Al\'ias and Palmer \cite{Alias} dealt with
curvature properties of such surfaces in $\mathbb{R}^4_1$.
The Weierstrass type representation in $\mathbb{R}^4_1$ should also be known already.
It seems quite natural to extend from minimal surfaces in $\mathbb{R}^3$
to our case by using the classical method systematically.
What puzzled and surprised us is that nobody did this before
(to the best of our knowledge). The main reasons might be as follows.

First, according to our observation, there do exist richer
phenomenon and new difficulties in $\mathbb{R}^4_1$. The Osserman's
theorem fails, and the singular ends as well as the problem of
solving the equation $\phi=\bar\psi$ are new challenges not so easy to
overcome.

Second, as pointed out at the beginning, the stability property of
stationary surfaces in $\mathbb{R}^4_1$ is bad. People might not
have great interest in considering a surface class without much
physical significance. (Yet under suitable restrictions on the allowed
variations, there are still stability results. See Palmer's work
\cite{Alias-Palmer},\cite{Palmer}.)

Last but not least, the embedding problem becomes more difficult, and in the 4-dimensional case
we lose the geometric intuition and the beautiful pictures which are always so attractive in the 3-space.
For the purpose of visualization, one may take a slice (i.e. the intersection of $\vec{x}(M)$
with a 3-space $x_4=0$ or $x_3=0$), or a projection to such 3-spaces.
But neither of them is satisfactory.\\

\textbf{Acknowledgement}~~
We thank the referee for helpful suggestions which improve our exposition.
This work is supported by the the Project 10771005 and the Project 10901006 of
National Natural Science Foundation of China.

\section{Weierstrass Representation}

Let $\vec{x}:M^2\to \mathbb{R}^4_1$ be an oriented complete spacelike
surface in 4-dimensional Lorentz space
with local complex coordinate $z=u+\mathrm{i}v$ and zero mean curvature $\vec{H}=0$.
The Lorentz inner product $\langle\cdot,\cdot\rangle$ is given by
\[\langle \vec{x},\vec{x}\rangle=x_1^2+x_2^2+x_3^2-x_4^2.\]
The induced Riemannian metric is $\mathrm{d}s^2=\mathrm{e}^{2\omega}|\mathrm{d}z|^2$ on $M$. Hence
\[
\langle \vec{x}_{z},\vec{x}_{z}\rangle=0,~~ \langle
\vec{x}_{z},\vec{x}_{\bar{z}}\rangle =\frac{1}{2}\mathrm{e}^{2\omega}.
\]
Choose null vectors $\vec{y},\vec{y}^*$ in the normal plane
at each point such that
\[
\langle \vec{y},\vec{y}\rangle=\langle \vec{y}^*,\vec{y}^*\rangle=0, ~~ \langle
\vec{y},\vec{y}^*\rangle =1,~~ \mathrm{det}\{\vec{x}_u,\vec{x}_v,\vec{y},\vec{y}^*\}>0~.
\]
Such frames $\{\vec{y},\ \vec{y}^{*}\}$ are
determined up to a scaling
\begin{equation}\label{scaling}
 \{\vec{y},\ \vec{y}^{*}\}\rightarrow
\{\lambda \vec{y},\ \lambda^{-1}\vec{y}^{*}\}
\end{equation}
for some non-zero real-valued function $\lambda$. After projection, we obtain two well-defined maps
\[
[\vec{y}],\ [\vec{y}^* ]: M \rightarrow S^2\cong\{[\vec{v}]\in\mathbb{R}P^3|\langle \vec{v},\vec{v}\rangle=0\}.
\]
The target space is usually called the projective light-cone, which is well-known
 to be homeomorphic to the 2-sphere. By analogy to $\mathbb{R}^3$,
 we call them the \emph{Gauss maps} of the spacelike surface $\vec{x}$ in $\mathbb{R}^{4}_{1}$.

The structure equations are derived under the complex frame $\{\vec{x}_z,\vec{x}_{\bar{z}},\vec{y},\vec{y}^*\}$:
\begin{flushleft}
\begin{equation}\label{xzz}\vec{x}_{zz}=2\omega_{z}~\vec{x}_z+\Omega^*~\vec{y}+\Omega
~\vec{y}^*,\end{equation}
\begin{equation}\label{xzbar}\vec{x}_{z\bar{z}}=0,\end{equation}
\begin{equation}\label{yz-s}\vec{y}_{z}=\alpha ~\vec{y} -2\mathrm{e}^{-2\omega}\Omega ~\vec{x}_{\bar{z}},\end{equation}
\begin{equation}\label{y*z-s}\vec{y}^*_{z}=-\alpha ~\vec{y}^* -2\mathrm{e}^{-2\omega}\Omega^* ~\vec{x}_{\bar{z}}.\end{equation}
\end{flushleft}
Here $\Omega\triangleq\langle \vec{x}_{zz},\vec{y}\rangle, \Omega^*\triangleq\langle \vec{x}_{zz},\vec{y}^*\rangle$
are only quasi-invariants (since they are defined only up to a scaling \eqref{scaling})
 whose geometric meaning is similar to the usual Hopf differential.
 $\alpha \mathrm{d}z\triangleq\langle \vec{y}_z,\vec{y}^*\rangle \mathrm{d}z$ gives the connection $1$-form of
 the normal bundle.

 Note that equation~\eqref{xzbar} means $\vec{x}$ is still a
 vector-valued harmonic function, which is equivalent to the condition $\vec{H}=0$
 (zero mean curvature).

 The integrability equations are:
\begin{gather}
K =-4\mathrm{e}^{-2\omega}\omega_{z\bar{z}}
=-4\mathrm{e}^{-4\omega}(\Omega\overline{\Omega^*}+\bar{\Omega}\Omega^*)~,
\label{Gauss}\\
\Omega_{\bar{z}} -\bar{\alpha}\Omega=\Omega^*_{\bar{z}}+\bar{\alpha}\Omega^*=0~,
\label{Codazzi}\\
K^{\bot} =-2\mathrm{i}\cdot \mathrm{e}^{-2\omega}(\alpha_{\bar{z}}-\bar{\alpha}_{z})
=-4\mathrm{i}\cdot \mathrm{e}^{-4\omega}(\Omega\overline{\Omega^*}-\bar{\Omega}\Omega^*)~.
\label{Ricci}
\end{gather}
Here $K,K^{\bot}$ denote the Gaussian curvature and the normal curvature, respectively.
The equations \eqref{Gauss} and \eqref{Ricci} may be combined into a single formula:
\begin{equation}\label{eq-curvature}
-K+\mathrm{i}K^{\perp}=8\mathrm{e}^{-4\omega}\Omega\overline{\Omega^*}~.
\end{equation}
Together with \eqref{Codazzi}, it follows that
\begin{equation}
\Delta\ln(-K+\mathrm{i}K^{\perp})=-4K+2K^{\perp}\mathrm{i}~,
\end{equation}
where $\Delta\triangleq 4\mathrm{e}^{-2\omega}\frac{\partial^2}{\partial z\partial {\bar z}}$ is the usual Laplacian operator with respect to $\mathrm{d}s^2$. (Similar formulas have appeared in \cite{Alias}.)

The assumption of zero mean curvature implies that the two Gauss maps
$[\vec{y}],\ [\vec{y}^* ]: M \rightarrow S^2$ are conformal, i.e.
\[
\langle \vec{y}_z,\vec{y}_z\rangle=\langle \vec{y}^*_z,\vec{y}^*_z\rangle=0~,
\]
by \eqref{yz-s} \eqref{y*z-s}. In particular, they induce opposite orientations on the
target space $S^2$ according to the observation in \cite{CPWang2008}.
Assume that $[\vec{y}]$ is given by a meromorphic function
$\phi=\phi^1+\mathrm{i}\phi^2:M\rightarrow \mathbb{C}\cup\{\infty\}$,
and $[\vec{y}^*]$ is given by the complex conjugation of a meromorphic function, i.e.
$\bar{\psi}=\psi^1-\mathrm{i}\psi^2:M\rightarrow \mathbb{C}\cup\{\infty\}$.

Since $[\vec{y}]\neq[\vec{y}^* ]$ at any point, we have $\phi\neq\bar{\psi }$ over $M$ and they do not have poles at the same regular point. Denote
\begin{equation}\label{eq-rho}
\phi-\bar{\psi}=\rho \mathrm{e}^{\mathrm{i}\theta},\ \rho>0~.
\end{equation}
Then we can write
\begin{equation}
\vec{y}=\frac{\sqrt{2}}{\rho}\left(\phi^1,\phi^2,\frac{1-|\phi|^2}{2},
\frac{1+|\phi|^2}{2}\right),
\end{equation}
\begin{equation}
\vec{y}^*=-\frac{\sqrt{2}}{\rho}\left(\psi^1,-\psi^2,\frac{1-|\psi|^2}{2},
\frac{1+|\psi|^2}{2}\right)~.
\end{equation}
Since $\phi_{\bar{z}}=\psi_{\bar{z}}=0$, by a direct calculation we have
\begin{equation}\alpha=\langle \vec{y}_z,\vec{y}^*\rangle=\mathrm{i}\theta_z,\end{equation}
\begin{equation}\label{yz}
\vec{y}_z-\alpha \vec{y}=-\frac{1}{\sqrt{2}\rho^2}
\mathrm{e}^{-\mathrm{i}\theta}\phi_z\Big(\bar{\phi}+\bar{\psi}, i(\bar{\phi}-\bar{\psi}),
1-\bar{\phi}\bar{\psi},1+\bar{\phi}\bar{\psi}\Big)~,
\end{equation}
\begin{equation}\label{y*z}
\vec{y}^*_z+\alpha \vec{y}^*=-\frac{1}{\sqrt{2}\rho^2}
\mathrm{e}^{\mathrm{i}\theta}\psi_z\Big(\bar{\phi}+\bar{\psi}, i(\bar{\phi}-\bar{\psi}),
1-\bar{\phi}\bar{\psi},1+\bar{\phi}\bar{\psi}\Big)~.
\end{equation}
Comparing \eqref{yz-s} with \eqref{yz},
we see that the holomorphic vector-valued 1-form $\vec{x}_z$ is  parallel to the meromorphic vector-valued function
\[
\Big(\phi+\psi, -\mathrm{i}(\phi-\psi),1-\phi\psi,1+\phi\psi\Big).
\]
Thus there exists a meromorphic differential $\mathrm{d}h$ such that
\begin{equation}\label{xz}
\vec{x}_z \mathrm{d}z=\Big(\phi+\psi, -\mathrm{i}(\phi-\psi),1-\phi\psi,1+\phi\psi\Big)\mathrm{d}h~.
\end{equation}
Observe that $\mathrm{d}h$ is indeed a holomorphic 1-form,
because it is the arithmetic average of the third and
the fourth components of $\vec{x}_z\mathrm{d}z$ by \eqref{xz}.
Although it is not necessarily an exact 1-form, according to the convention in this research field of minimal surfaces we still denote it as $\mathrm{d}h$ and call it the \emph{height differential} (see footnote~8 in \cite{Meeks3}).

Finally, we derive a Weierstrass
representation of stationary surface $\vec{x} : M\rightarrow \mathbb{R}^4_1:$
\begin{equation}\label{x}
\vec{x}= 2~\mathrm{Re}\int \Big(\phi+\psi, -\mathrm{i}(\phi-\psi),1-\phi\psi,1+\phi\psi\Big)\mathrm{d}h.
\end{equation}
in terms of two meromorphic functions $\phi,\psi$ and
a holomorphic $1$-form $\mathrm{d}h$.

\begin{remark}\label{rem-Wrepre}
When $\phi\equiv -1/\psi$, the above formula \eqref{x} yields a minimal
surface in $\mathbb{R}^3$ and we recover the classical Weierstrass representation.
When $\phi\equiv 1/\psi$, one obtains the Weierstrass representation for a maximal
surface in $\mathbb{R}^3_1$.
When $\phi$ or $\psi$ is constant, without loss of generality
(see Remark~\ref{rem-trans} and \eqref{trans}) we may assume
$\psi\equiv 0$. After integration we get $x_3-x_4=constant$, hence $\vec{x}$ is
a zero mean curvature surface in a 3-space $\mathbb{R}^3_0$
(which is endowed with a degenerate inner product).
Thus all these classical cases are included as special
cases of our generalized Weierstrass type representation.
\end{remark}
\begin{definition}
Similar to the case of minimal surfaces in
$\mathbb{R}^3$, we call $\phi,\psi$ \emph{the Gauss maps} of $\vec{x}$, and
$\mathrm{d}h$ \emph{the height differential}.
\end{definition}
\begin{remark}\label{rem-trans}
It is important to consider the effect of a Lorentz isometry of $\mathbb{R}^4_1$
on the Weierstrass data, which will be frequently used to simplify
the construction of examples, and to reduce the general situation to
some special cases. Observe that the induced action on the projective light-cone
is nothing but a M\"obius transformation on $S^2$, or just
a fractional linear transformation on $\mathbb{C}P^1=\mathbb{C}\cup\{\infty\}$
given by $A=\left(\begin{smallmatrix}a & b \\ c & d\end{smallmatrix}\right)$
 with $a,b,c,d \in \mathbb{C},~ad-bc=1$.
The Gauss maps $\phi,\psi$ and the height differential $\mathrm{d}h$ transform as below:
\begin{equation}\label{trans}
\phi\Rightarrow
\frac{a\phi+b}{c\phi+d}~,~~
 \psi\Rightarrow \frac{\bar{a}\psi+\bar{b}}{\bar{c}\psi+\bar{d}}~,~~
 \mathrm{d}h\Rightarrow (c\phi+d)(\bar{c}\psi+\bar{d})\mathrm{d}h~.
 \end{equation}
 \end{remark}

\begin{theorem}\label{thm-period}
Given a holomorphic $1$-form $\mathrm{d}h$ and two meromorphic functions $\phi,\psi:M\rightarrow \mathbb{C}\cup\{\infty\}$
globally defined on a Riemann surface $M$, assume that they satisfy the regularity conditions
(1),(2) and the period condition (3) as below:

(1) $\phi\neq\bar{\psi}$ on $M$ and their poles do not coincide;

(2) The zeros of $\mathrm{d}h$ coincide with the poles of $\phi$ or $\psi$
with the same order;

(3) Along any closed path the periods satisfy
\begin{equation}\label{eq-period}
\oint_\gamma \phi \mathrm{d}h
=-\overline{\oint_\gamma \psi \mathrm{d}h }, ~~
\mathrm{Re}\oint_\gamma \mathrm{d}h=0=\mathrm{Re}\oint_\gamma \phi\psi \mathrm{d}h~.
\end{equation}
Then
\eqref{x} defines a stationary surface $\vec{x}:M\rightarrow \mathbb{R}^4_1$.

Conversely, every stationary surface $\vec{x}:M\rightarrow \mathbb{R}^4_1$ can be
represented as \eqref{x} in terms of such $\phi,\ \psi$ and $\mathrm{d}h$ over a (necessarily non-compact) Riemann surface $M$.
\end{theorem}

Instead of proving the theorem, which is easy and similar to the case of minimal surfaces in $\mathbb{R}^3$, we provide a detailed explanation of conditions (1), (2).
For a stationary surface constructed by \eqref{xz}, we find the metric
\begin{equation}\label{eq-metric}
\mathrm{e}^{2\omega}|\mathrm{d}z|^2
=\langle \vec{x}_z,\vec{x}_{\bar{z}}\rangle|\mathrm{d}z|^2
=2|\phi-\bar{\psi}|^2|\mathrm{d}h|^2~.
\end{equation}
Thus $\phi\neq\bar{\psi}$ when the values $\phi(z),\psi(z)$ are finite.
The exceptional case is at the poles of $\phi$ or $\psi$,
which is equivalent to the previous situation, since
we could take any point on $S^2$ to be the north pole up a rotation,
or equivalently, take any point to be $\infty\in \mathbb{C}P^1$
up to a linear fractional transformation as in \eqref{trans}.
So $\phi,\psi$ would not have poles at the same point on $M$.
This explains the condition (1). Condition (2) is self-evident by \eqref{eq-metric}.
(The geometric interpretation for $\phi\neq\bar{\psi}$ is that
$[\vec{y}],[\vec{y}^*]$ are distinct at each regular point as pointed out before \eqref{eq-rho}.)

In $\mathbb{R}^3$, condition (1) is satisfied automatically since
we have $\phi=-\frac{1}{\psi}$ which is never equal to $\bar\psi$.
The two lightlike normal vectors corresponding to them are given by
$(\vec{n},\pm 1)\in \mathbb{R}^4_1$ where $\vec{n}$ is
the unit normal vector in $\mathbb{R}^3$.

By contrast, a typical global phenomenon for maximal surfaces in $\mathbb{R}^3_1$ is the existence of singularities \cite{Kobayashi, Umehara}.
The best-known example is \emph{the catenoid in $\mathbb{R}^3_1$}
given by \eqref{x} with
\begin{equation}\label{eq-catenoid2}
\phi=z,~~\psi=\frac{1}{z},~~\mathrm{d}h=\frac{\mathrm{d}z}{z},~~
M=\mathbb{C}\backslash\{0\}.
\end{equation}
By integration we obtain a rotational
maximal surface $\vec{x}=(x_1,x_2,x_3,x_4)$ with
\[
\sqrt{x_1^2+x_2^2}=\sinh(x_4),~~x_3=0.
\]
Compared to the catenoid in $\mathbb{R}^3$,
this example is peculiar in that it has a cone-like singularity,
which corresponds to the locus $|z|=1$
where $\phi=\bar\psi$.\\

Next let us express the geometrical quantities of stationary surface
$\vec{x}:M\rightarrow \mathbb{R}^4_1$ in terms of the Weierstrass data
$\phi,\psi$ and $\mathrm{d}h=h'(z)\mathrm{d}z$.
By \eqref{yz-s}, \eqref{y*z-s}, \eqref{yz}, \eqref{y*z},
\eqref{xz}, and \eqref{eq-metric} we obtain
\begin{equation}
\Omega=\frac{1}{\sqrt{2}}\mathrm{e}^{-\mathrm{i}\theta}h'(z)\phi_z,\
\Omega^*=\frac{1}{\sqrt{2}}\mathrm{e}^{\mathrm{i}\theta}h'(z)\psi_z~.
\end{equation}
Together with \eqref{eq-curvature} we have
\begin{equation}
-K+\mathrm{i}K^{\perp}
=4\mathrm{e}^{-4\omega}\mathrm{e}^{-2\mathrm{i}\theta}|h'(z)|^2\phi_z\bar{\psi}_{\bar{z}}
=4\mathrm{e}^{-2\omega}\frac{\phi_z\bar{\psi}_{\bar{z}}}
{(\phi-\bar{\psi})^2}
=\Delta\ln(\phi-\bar\psi)~.
\end{equation}
So we derive an extremely important formula for the total Gaussian and
normal curvature over a compact stationary surface $M$ with boundary $\partial M$:
\begin{equation}
\begin{split}
\int_M(-K+\mathrm{i}K^{\perp})\mathrm{d}M&=
\int_M \frac{4\phi_z\bar{\psi}_{\bar{z}}}{(\phi-\bar{\psi})^2}
\mathrm{d}u\wedge \mathrm{d}v =2\mathrm{i}\int_M
\frac{\phi_z\bar{\psi}_{\bar{z}}}{(\phi-\bar{\psi})^2}
\mathrm{d}z\wedge \mathrm{d}\bar{z}
\\
&=-2\mathrm{i}\int_{\partial M} \frac{\phi_z}{\phi-\bar{\psi}} \mathrm{d}z
=-2\mathrm{i}\int_{\partial M}
\frac{\bar{\psi}_{\bar{z}}}{\phi-\bar{\psi}}\mathrm{d}\bar{z}.
\label{eq-totalcurvature}
\end{split}
\end{equation}

\begin{remark}\label{rem-Ksign}
It is easy to see that $K^{\bot}\equiv 0$ when $\vec{x}$ is contained
in a 3-dimensional subspace. By Remark~\ref{rem-Wrepre} we deduce that
\begin{itemize}
\item $K\le 0$ for minimal surfaces in $\mathbb{R}^3$;
\item $K\ge 0$ for maximal surfaces in $\mathbb{R}^3_1$;
\item $K\equiv 0$ for zero mean curvature surfaces in $\mathbb{R}^3_0$.
\end{itemize}
These facts are well-known. For more on curvature properties
of complete stationary surfaces, see \cite{Alias}.
We emphasize that in general the Gaussian curvature $K$
does not have a fixed sign.
Hence the improper integral $\int_M K\mathrm{d}M$ is meaningful
only when it is \emph{absolutely convergent}.
\end{remark}
\begin{remark}\label{rem-totalK}
The complex integral $\int_M(-K+\mathrm{i}K^{\perp})\mathrm{d}M$ is very important.
We say that $\vec{x}:M\to \mathbb{R}^4_1$ has \emph{finite total curvature} if
this integral converges absolutely, i.e.,
\[\int_M|-K+\mathrm{i}K^{\perp}|\mathrm{d}M<+\infty\]
This implies finite total Gaussian curvature $\int_M|K|\mathrm{d}M<+\infty$.
But we do not know whether the converse is true.
\end{remark}
\medskip
\noindent \textbf{Convention:} In this paper, we always assume that
neither of $\phi,\psi$ is a constant unless it is stated otherwise.
According to Remark~\ref{rem-Wrepre}, that means we have ruled out
the trivial case of stationary surfaces in $\mathbb{R}^3_0$.
We always consider oriented stationary surfaces. The theory and examples of non-orientable stationary surfaces are treated in another paper \cite{Ma-4pi}.

\section{Finite total curvature and essential singularities}

Recall that a significant class of minimal surfaces in $\mathbb{R}^3$
is those complete ones with finite total Gaussian curvature, i.e.,
\[
\int_M|K|\mathrm{d}M=-\int_M K\mathrm{d}M<+\infty.
\]
The importance of this condition relies on the following classical result.
\begin{theorem} Let $(M,\mathrm{d}s^2)$ be a non-compact surface with a complete metric. Suppose $\int_M|K|\mathrm{d}M<+\infty$. Then we have the following conclusion.

(1) (Huber\cite{Huber}) There is a compact Riemann surface
$\overline{M}$ such that $M$ as a Riemann surface is conformal to $\overline{M}\backslash\{p_1,p_2,\cdots,p_r\}$.

(2) (Osserman\cite{Osser}) When $M$ is a minimal surface in $\mathbb{R}^3$ with the induced metric $\mathrm{d}s^2$, the Gauss map $G=\phi=-1/\psi$ and the
height differential $\mathrm{d}h$ extend to each end $p_j$ analytically.

(3) (Jorge and Meeks \cite{Jor-Meeks}) As in (1) and (2), suppose minimal surface $M\to \mathbb{R}^3$ has $r$ ends and $\overline{M}$ is the compactification with genus $g$. The total curvature is related with these topological invariants via the Jorge-Meeks formula:
\begin{equation}\label{eq-jorgemeeks}
\int_M
K\mathrm{d}M=2\pi\left(2-2g-r-\sum_{j=1}^r d_j\right)~,
\end{equation}
Here $d_j+1$ equals to the highest order of the pole of $\vec{x}_z \mathrm{d}z$ at $p_j$,
and $d_j$ is called \emph{the multiplicity at the end $p_j$}.
\end{theorem}
Huber's conclusion (1) means that
\[
\text{finite total curvature} ~~\Longrightarrow~~ \text{finite topology},
\]
which is a purely intrinsic result. In particular, this is valid also for stationary surfaces in $\mathbb{R}^4_1$.

As to the extrinsic geometry of a minimal surface $M^2\to \mathbb{R}^3$ with finite total curvature, Osserman's result (2) shows that we have a nice control over its behavior at infinity. To our surprise, this is no longer true in $\mathbb{R}^4_1$. In particular we have counterexamples given below:

\begin{example}
[$M_{k,a}$ with essential singularities and finite total curvature]\label{exa-essen}
The Weierstrass data is given by
\begin{equation}\label{ex-essential}
M_{k,a}\cong \mathbb{C}-\{0\},~\phi=z^k \mathrm{e}^{az},~\psi=-\frac{\mathrm{e}^{az}}{z^k},~\mathrm{d}h=\mathrm{e}^{-az}\mathrm{d}z~.
\end{equation}
where integer $k$ and real number $a$ satisfy
$k\ge 2, 0< a<\frac{\pi}{2}$.
\end{example}
\begin{proposition}\label{prop-essential}
Each of the stationary surfaces $M_{k,a}$ in Examples~\ref{exa-essen} is a regular, complete stationary surface with two ends at $z=0,\infty$ satisfying the period conditions. Moreover the total curvature converges absolutely with
\begin{equation}\label{eq-essential}
\int_M
K\mathrm{d}M=-4\pi k~, ~~\int_M K^{\perp}\mathrm{d}M =0~.
\end{equation}
\end{proposition}
\begin{proof}
First let us verify the regularity condition. The zero of $\phi$
and the pole of $\psi$ coincide at $z=0$ with the same multiplicity. At $z=\infty$, $\phi$ and $\psi$
have essential singularities. Except these two points,
$\phi,\psi$ are holomorphic nonzero functions, and we need
only to show $\phi\ne\bar{\psi}$ on $\mathbb{C}-\{0\}$.
Suppose $\phi(z)=\overline{\psi(z)}$ for some $z\ne 0$.
Using \eqref{ex-essential} and comparing the norms we see $|z|=1$.
Let $z=\mathrm{e}^{\mathrm{i}\theta}$ for some $\theta\in[0,2\pi)$.
Then the equation $\phi(z)=\overline{\psi(z)}$ is reduced to
\[\mathrm{e}^{2a\mathrm{i}\sin\theta}=-1,\]
which has no real solutions for $\theta$ when $0<a<\frac{\pi}{2}.$
This proves our first assertion.

Second, the period conditions are obviously satisfied, since any of
\[\mathrm{d}h=\mathrm{e}^{-az}\mathrm{d}z,~~\phi\psi \mathrm{d}h=-\mathrm{e}^{az}\mathrm{d}z,~~\phi \mathrm{d}h=z^k \mathrm{d}z,~~\psi \mathrm{d}h=-z^{-k}\mathrm{d}z\]
have no residues (note that $k\ge 2$).

Third, the metric of $M_{k,a}$ is complete by the following simple estimation:
\begin{equation}\label{eq-essen-metric}
\mathrm{d}s=|\phi-\bar\psi||\mathrm{d}h|=|z^k+(\bar{z})^{-k}\mathrm{e}^{a(\bar{z}-z)}||\mathrm{d}z|\sim
\left\{\begin{array}{rl}
|z|^{-k}|\mathrm{d}z|,& (z\to 0);\\
|z|^k|\mathrm{d}z|,& (z\to\infty).
\end{array}\right.
\end{equation}

Finally we estimate the absolute total curvature:
\begin{align*}
\int_{\mathbb{C}}|-K+\mathrm{i}K^{\perp}|\mathrm{d}M&=2\mathrm{i}\int_{\mathbb{C}}
\frac{|\phi_z\bar{\psi}_{\bar{z}}|}{|\phi-\bar{\psi}|^2}
\mathrm{d}z\wedge \mathrm{d}\bar{z}\\
&=2\mathrm{i}\int_{\mathbb{C}}
\frac{|az+k||k-az|}{\left|z^{k+1}+
z\bar{z}^{-k}\mathrm{e}^{a(\bar{z}-z)}\right|^2}
\mathrm{d}z\wedge \mathrm{d}\bar{z}.
\end{align*}
When $r=|z|\to 0$ the integrand is almost the same as $k^2 r^{2k-1}\mathrm{d}r \mathrm{d}\theta$
with respect to the polar coordinate; when $r=|z|\to\infty$ we have a similar
estimation as $r^{1-2k}\mathrm{d}r \mathrm{d}\theta$. Since $k\ge 2$, each of these two improper integral
converges when $r\to 0$ or $r\to \infty$, respectively.
Thus the total curvature integral converges absolutely.

To find the exact value of the integration, we approximate $\mathbb{C}$ by domains
$A_{r,R}\triangleq\{0<r\le |z|\le R\}$. Using Stokes theorem for $A_{r,R}$
we obtain
\begin{align*}
\int_{A_{r,R}}(-K+\mathrm{i}K^{\perp})\mathrm{d}M
&=2\mathrm{i}\int_{A_{r,R}}
\frac{\phi_z\bar{\psi}_{\bar{z}}}{(\phi-\bar{\psi})^2}
\mathrm{d}z\wedge \mathrm{d}\bar{z}\\
&=-2\mathrm{i}\oint_{|z|=R}
\frac{\phi_z}{\phi-\bar{\psi}}\mathrm{d}z+2\mathrm{i}\oint_{|z|=r}
\frac{\phi_z}{\phi-\bar{\psi}}\mathrm{d}z~,
\end{align*}
where
\[\frac{\phi_z}{\phi-\bar{\psi}}=\frac{a+k/z}{1+\mathrm{e}^{a(\bar{z}-z)}/|z|^{2k}}~.
\]
When $R\to\infty$ the first contour integral converges to $-2\mathrm{i}\oint_{|z|=R}\frac{k}{z}\mathrm{d}z=4\pi k$. When $r\to 0$ the second contour integral converges to $0$. This completes the proof to Proposition~\ref{prop-essential} as well as the formula \eqref{eq-essential}.
\end{proof}

We can construct similar examples as below. The reader may compare
this to the generalized catenoid in Section~7 and the generalized
Enneper-type surfaces in Section~8.
\begin{example}
[The Enneper surface $E_{k,a}$ with essential singularities]
\label{exa-essen2}:
\[
E_{k,a}\cong \mathbb{C},~\phi=z^k \mathrm{e}^{az},~\psi=-\frac{\mathrm{e}^{az}}{z^k},
~\mathrm{d}h=z^k \mathrm{e}^{-az}\mathrm{d}z~.
~~(k\ge 2,~0<a<\pi/2)
\]
\end{example}
\begin{example}
[The catenoid $C_{k,a}$ with essential singularities]
\label{exa-essen3}:
\[
C_{k,a}\cong \mathbb{C}-\{0\},~\phi=z^k
\mathrm{e}^{az},~\psi=-\frac{\mathrm{e}^{az}}{z^k},~ \mathrm{d}h=\frac{\mathrm{e}^{-az}}{z^k}\mathrm{d}z~.
~~(k\ge 2,~0<a<\pi/2)
\]
\end{example}
\begin{remark}
When $a=0$ in Example~\ref{exa-essen2} we obtain the \emph{Enneper surfaces
in $\mathbb{R}^3$ with higher dihedral symmetry}. Similarly in
Example~\ref{exa-essen3} when $a=0,k=1$ we get the classical
catenoid in $\mathbb{R}^3$. So these examples might be regarded as
deformations of Enneper-type surfaces and the catenoid.
When $k=1$ and $a\ne 0$, it seems that Example~\ref{exa-essen2} and
\ref{exa-essen3} provide new complete regular stationary surfaces in
$\mathbb{R}^4_1$ with total curvature $\int_M K\mathrm{d}M=-4\pi.$ Yet this is incorrect since the total Gaussian curvature does not
converge absolutely in this case (see the discussion after \eqref{eq-essen-metric}).
\end{remark}

In general, a minimal surface in $\mathbb{R}^3$ or
a stationary surface in $\mathbb{R}^4_1$ is said to be
\emph{of algebraic type} if there exists a
compact Riemann surface $\overline{M}$
with $M=\overline{M}\backslash\{p_1,p_2,\cdots,p_r\}$ such that
$\vec{x}_z \mathrm{d}z$ is a vector valued meromorphic form defined
on $\overline{M}$. In other words, the Gauss map $\phi,\psi$
and height differential $\mathrm{d}h$ extend to meromorphic functions/forms
on $\overline{M}$. For this surface class we may
establish a Gauss-Bonnet type formula similar to \eqref{eq-jorgemeeks}
(see Theorem~\ref{GB3}) which involves the indices of
the so-called singular ends. This is discussed in the next two sections.

\section{Singular ends}

Let $\vec{x}:M\rightarrow \mathbb{R}^4_1$ be a complete stationary surface
given by \eqref{x} with Weierstrass data $\phi,\psi,\mathrm{d}h$.
Its global geometry and topology is closely related with the behavior at infinity,
where we have a similar definition of \emph{(annular) ends} like the case in $\mathbb{R}^3$.
In this paper we always assume that the period condition \eqref{eq-period}
is satisfied at an annular end unless it is stated otherwise.

At an end of $\vec{x}:M\rightarrow \mathbb{R}^4_1$, although one of the four components of $\vec{x}$
tends to $\infty$, the surface might still be incomplete due to the Lorentz-type metric
of the ambient space. The simplest example is as below:
\begin{equation}\label{eq-incomplete}
\phi=z^{-2},~\psi=z^{-3},~\mathrm{d}h=\mathrm{d}z.
\end{equation}
Using \eqref{x} \eqref{eq-metric}, the reader can verify that
the surface has an incomplete end at $z=\infty$. Notice that at this end we have
$\phi(\infty)=\bar\psi(\infty)$
which is a typical case for such examples.

In general, although $\phi\ne\bar\psi$ at any regular point of a stationary $M\to\mathbb{R}^4_1$, it might happen that
$\phi=\bar\psi$ at one of the ends. In this case the total curvature
\eqref{eq-totalcurvature} involves an improper integral. On the other hand,
when the total curvature is finite, the worst case is that there are finite many such ends by Huber's theorem.
Thus we restrict to consider isolated zeros of complex harmonic function $\phi-\bar\psi$ in this section.

(It is an annoying problem to solve the equation $\phi(z)-\bar\psi(z)=0$ or to show non-existence of solutions for given functions $\phi,\psi$ over a given Riemann surface $M$. See Section~6 to 9 for examples and discussions.)

Below we always use $D$ to denote a neighborhood of $z=0$ on $\mathbb{C}$ and
$D$ is homeomorphic to a disc. By $D\to \{0\}$ we mean any limit process,
or a sequence of smaller and smaller neighborhood
$\cdots\supset D_i\supset D_{i+1}\supset\cdots$ whose intersection is $\{0\}$.

\subsection{Two types of singular ends}
\begin{definition}
Let $\vec{x}:D-\{0\}\rightarrow \mathbb{R}^4_1$
be an annular end of a regular stationary surface (with boundary)
whose Gauss maps $\phi$ and $\psi$
extend to meromorphic functions on $D$ (namely, $z=0$
is a removable singularity or a pole for both $\phi, \psi$).
It is called a \emph{regular end} when
\[
\phi(0)\ne\bar\psi(0).~~~(\text{Thus}~ \phi(z)\ne\bar\psi(z),~\forall~z\in D.)
\]
It is a \emph{singular end} if $\phi(0)=\bar\psi(0)$
where the value could be finite or $\infty$ (i.e., $z=0$ is a pole of both $\phi$ and $\psi$).
Such ends are divided into two classes
depending on whether functions $\phi,\bar\psi$
take the same value at $z=0$ with the same multiplicity or not.
When the multiplicities are equal we call it a \emph{bad singular end}. Otherwise it is called a \emph{good singular end}.
\end{definition}
The total Gaussian curvature is finite around a regular end by \eqref{eq-totalcurvature}.
As to a singular end, we may assume that $\phi=\bar{\psi}=0$ at $z=0$.
Without loss of generality we may write it out more explicitly:
\[
\phi=z^m+o(z^m),~\psi=b z^n+o(z^n)
\]
for complex number $b\ne 0$ and positive integers $m,n$.
Write $z=r\mathrm{e}^{\mathrm{i}\theta}$ with $r\ge 0$ and $\theta\in[0,2\pi)$.
At a good singular end $m\neq n$. Suppose $m>n\ge 1$. Then by \eqref{eq-totalcurvature},
\begin{align*}
&\int_{D}(-K+\mathrm{i}K^{\perp})\mathrm{d}M=2\mathrm{i}\int_{D}
\frac{\phi_z\bar{\psi}_{\bar{z}}}{(\phi-\bar{\psi})^2}\mathrm{d}z\wedge \mathrm{d}\bar{z}\\
=-4&\int_{D}\frac{r^{m+n-2}(\mathrm{e}^{\mathrm{i}(m-1)\theta}+o(r))(\mathrm{e}^{-\mathrm{i}(n-1)\theta}+o(r))}
{r^{2n}(b\mathrm{e}^{-\mathrm{i}(n-1)\theta}+o(r)))^2}r \mathrm{d}r \mathrm{d}\theta\rightarrow 0.~~(
D\rightarrow \{0\})
\end{align*}
This integral converges absolutely.
\begin{remark}
\label{rem-limit}
In particular, by \eqref{eq-totalcurvature} we see that the line integral
\[
\oint_{\partial D} \frac{\phi_z}{\phi-\bar{\psi}} \mathrm{d}z
~~\text{and}~~\oint_{\partial D}
\frac{\bar{\psi}_{\bar{z}}}{\phi-\bar{\psi}}\mathrm{d}\bar{z}
\]
both have well-defined limit when the neighborhood
$D\to \{0\}$.
\end{remark}

If this is a bad singular end, $m=n\ge 1$, then
\[
\int_{D}(-K+\mathrm{i}K^{\perp})\mathrm{d}M
=-4\int_{D}\left(\frac{f(\mathrm{e}^{\mathrm{i}\theta})}{r}+o\left(\frac{1}{r}\right)\right)
\mathrm{d}r \mathrm{d}\theta~.
\]
This integration does not converge absolutely.
Moreover it does depend on the limit process of $D\to \{0\}$.
As an explicit example, one can verify that the integral
\[
\int_{\partial A}\frac{\mathrm{d}z}{z-2\bar{z}},~~~
A=\{z|\mathrm{Re}(z)\in(-a,a),\mathrm{Im}(z)\in(-b,b)\}
\]
does depend on the choice of positive real numbers $a,b$.
Hence by \eqref{eq-totalcurvature} the curvature integral when
 $\phi(z)=z,\psi(z)=2z$ is divergent around
 any neighborhood of $z=0$.

 In summary we have proved
\begin{proposition}\label{prop-goodsingular}
A singular end of a stationary surface $\vec{x}:D-\{0\}\rightarrow \mathbb{R}^4_1$
is \emph{good} if and only if the curvature integral
\eqref{eq-totalcurvature} converges absolutely around this end.
Around a regular end or a good singular end, the two line integrals $\oint_{\partial D} \frac{\phi_z}{\phi-\bar{\psi}} \mathrm{d}z$
and $\oint_{\partial D}
\frac{\bar{\psi}_{\bar{z}}}{\phi-\bar{\psi}}\mathrm{d}\bar{z}$ has a well-defined limit when $D\to \{0\}$.
\end{proposition}

Later we will see that there exist complete stationary surfaces
of finite total curvature with finite many good singular ends
(see Example~\ref{exa-singular1} and \ref{exa-singular2}).

\subsection{Index of a good singular end}

To establish a Gauss-Bonnet type formula relating the total Gaussian curvature
and the topology for such surfaces we need to define the index of a good singular end.
This is equivalent to the usual index of a zero of the complex function $\phi-\bar\psi$
(suppose at the end $z=0$, $\phi(0)=\bar\psi(0)\ne\infty$ without loss of generality).
\begin{lemma}
\label{lemma-index}
Denote $D_{\varepsilon}=\{z||z|<\varepsilon\}$. Let $m,n$ be non-negative integers. Then
\begin{equation}
\label{eq-index}
\oint_{\partial D_{\varepsilon}}d\ln(z^m-\bar{z}^n)=\left\{
   \begin{array}{ll}
          2\pi\mathrm{i}\cdot m, & \hbox{$m<n$;} \\
    -2\pi\mathrm{i}\cdot n, & \hbox{ $m>n$.}
   \end{array}
 \right.
 \end{equation}
 \end{lemma}
\begin{proof}
The result is trivial when $m=0$ or $n=0$. If $m,n\ge 1$, we have
\[
\oint_{\partial D_{\varepsilon}}d\ln(z^m-\bar{z}^n)=
 \oint_{\partial D_{\varepsilon}}\frac{mz^{m-1}}{z^m-\bar{z}^n}\mathrm{d}z+
 \oint_{\partial D_{\varepsilon}}\frac{-n\bar{z}^{n-1}}{z^m-\bar{z}^n}\mathrm{d}\bar{z}~.
 \]
\[
\text{If}~m>n,~~\lim_{\varepsilon\rightarrow0} \oint_{\partial
D_{\varepsilon}}\frac{mz^{m-1}}{z^m-\bar{z}^n}\mathrm{d}z=0~,~~~
\lim_{\varepsilon\rightarrow0} \oint_{\partial
D_{\varepsilon}}\frac{-n\bar{z}^{n-1}}{z^m-\bar{z}^n}\mathrm{d}\bar{z}=-n\cdot2\pi\mathrm{i}~.
\]
\[
\text{If}~m<n,~~
\lim_{\varepsilon\rightarrow0} \oint_{\partial
D_{\varepsilon}}\frac{mz^{m-1}}{z^m-\bar{z}^n}\mathrm{d}z=m\cdot2\pi\mathrm{i}~,~~~
\lim_{\varepsilon\rightarrow0} \oint_{\partial
D_{\varepsilon}}\frac{-n\bar{z}^{n-1}}{z^m-\bar{z}^n}\mathrm{d}\bar{z}=0~.
\]
Since the integral is independent to the choice of $D_\epsilon$,
taking the limit $\epsilon\to 0$ we obtain the result.
\end{proof}
Observe that according to Remark~\ref{rem-limit}, all these integrals above
has a well-defined limit  when we consider generic neighborhood $D_p$ and
general limit process $D_p\to\{p\}$. Moreover we can state our results in a
more general way as below and the proof is direct.
\begin{lemma}\label{lemma-index2}
Suppose $p$ is an isolated zero of $\phi-\bar\psi$ in $p$'s neighborhood $D_p$,
where holomorphic functions $\phi$ and $\psi$ take the value
$\phi(p)=\overline{\psi(p)}$ with multiplicity $m$ and $n$, respectively. Then we have:
\begin{gather*}
\text{If}~m>n,~~\lim_{D_p\to \{p\}} \oint_{\partial
D_p}\frac{\phi_z}{\phi-\bar\psi}\mathrm{d}z=0~,~~~
\lim_{D_p\to \{p\}} \oint_{\partial
D_p}\frac{\bar\psi_{\bar{z}}}{\phi-\bar\psi}\mathrm{d}\bar{z}=-n\cdot2\pi\mathrm{i}~.\\
\text{If}~m<n,~~
\lim_{D_p\to \{p\}} \oint_{\partial
D_p}\frac{\phi_z}{\phi-\bar\psi}\mathrm{d}z=m\cdot2\pi\mathrm{i}~,~~~
\lim_{D_p\to \{p\}} \oint_{\partial
D_p}\frac{\bar\psi_{\bar{z}}}{\phi-\bar\psi}\mathrm{d}\bar{z}=0~.
\end{gather*}
In particular,
\begin{equation}
\frac{1}{2\pi\mathrm{i} }\oint_{\partial D_p}d\ln(\phi-\bar{\psi})=\left\{
   \begin{array}{ll}
          m, & \hbox{$m<n$;} \\
         -n, & \hbox{ $m>n$.}
   \end{array}
 \right.
\end{equation}
\end{lemma}

These integer-valued topological invariants enable us to define
two kinds of indices at a good singular end $\vec{x}:D_p-\{p\}\rightarrow \mathbb{R}^4_1$.
\begin{definition}
\emph{The index of $\phi-\bar{\psi}$ at $p$} (when $\phi,\psi$ are both holomorphic at $p$) is
\begin{equation}\label{eqind}
\mathrm{ind}_p(\phi-\bar{\psi})\triangleq\lim_{D_{p}\rightarrow
\{p\}}\frac{1}{2\pi\mathrm{i}}\oint_{\partial D_{p}}d\ln(\phi-\bar{\psi})~.
\end{equation}
\emph{The absolute index of $\phi-\bar{\psi}$ at $p$} is
\begin{equation}\label{eqind+}
\mathrm{ind}^{+}_p(\phi-\bar{\psi})\triangleq\left|\mathrm{ind}_p(\phi-\bar{\psi})\right|.
\end{equation}
\end{definition}
\begin{remark}
For a regular end our index is still meaningful with the value
\[
\mathrm{ind}=\mathrm{ind}^+=0.
\]
So singular ends are distinguishes from regular ends by these indices.
\end{remark}
\begin{remark}
Note that our definition of index of $\phi-\bar\psi$ is invariant
under the action of fractional linear transformation \eqref{trans}.
So it is well-defined for a stationary surface with good singular ends
and independent to the choice of coordinates of $\mathbb{R}^4_1$.
In particular, we can always assume that our singular ends do not coincide with poles of $\phi,\psi$,
hence the definition above is valid.
\end{remark}
We notice that the poles of $\phi$ or $\psi$ are also
singularities of $d\ln(\phi-\bar\psi)$ in a general sense.
As a supplement to the definition above, we count their contribution to
the indices of singularities according to
\begin{lemma}\label{lem-pole}
Suppose $\phi$ is a meromorphic function in a
neighborhood $D_q$ of $q$ with one pole of order $k$, $\psi$ is holomorphic
in $D_q$. Then
\begin{gather*}
\oint_{\partial
D_{q}}d\ln(\phi-\bar{\psi})=
\lim_{D_{q}\rightarrow \{q\}}\oint_{\partial
D_{q}}\frac{\phi_z}{\phi-\bar{\psi}}\mathrm{d}z
= -2\pi\mathrm{i}\cdot k~,\\
\lim_{D_{q}\rightarrow \{q\}}\oint_{\partial
D_{q}}\frac{\bar\psi_{\bar{z}}}{\phi-\bar{\psi}}d{\bar{z}}
=0~.
\end{gather*}
(When $\phi$ is holomorphic and $\psi$ has a pole we have a similar result with a different sign.)
\end{lemma}
\begin{proof}
By the residue theorem, and simple estimation using
\[
\frac{\phi_z}{\phi-\bar{\psi}}-\frac{\phi_z}
{\phi-\overline{\psi(q)}}
=\frac{(\bar\psi-\overline{\psi(q)})\phi_z}
{(\phi-\bar{\psi})(\phi-\overline{\psi(q)})}~.~~~~~~~~~~~\qedhere
\]
\end{proof}

\subsection{An index theorem}

\begin{proposition}\label{prop-index}
Let $\phi,\psi:\overline{M}\to\mathbb{C}\cup\{\infty\}$ be meromorphic functions
on a compact Riemann surface $\overline{M}$. Suppose that $\phi=\bar{\psi}$ on $\{p_j,j=1,\cdots,k\}$.
Denote the degree of a holomorphic mapping $\phi$ by $\deg\phi$. Then we have
\begin{equation}
\sum{_{j}}\mathrm{ind}_{p_j}(\phi-\bar{\psi})=\deg\phi-\deg\psi~.
\end{equation}
\end{proposition}
\begin{proof}
Without loss of generality, suppose that $\phi,\psi$ have distinct poles $\{q_l\}$
and $\{\hat{q}_m\}$, whose orders sum to $\deg\phi$
and $\deg\psi$, respectively.
In particular, $p_j$'s are distinct from them.
$d\ln(\phi-\bar{\psi})$ is an exact $1$-form on the supplement of
these distinct points. By Stokes formula we get
\begin{equation*}
\begin{split}
2\pi\mathrm{i}\cdot & \sum_{j=1}^k \mathrm{ind}_{p_j}(\phi-\bar{\psi})
=\sum_{j=1}^{k}\lim_{D_{p_j}\rightarrow \{p_j\}}
\oint_{\partial D_{p_j}}d\ln(\phi-\bar{\psi})\\
&=-\sum_{l=1}^{k_1}\lim_{D_{q_l}\rightarrow \{q_l\}}
\oint_{\partial D_{q_l}}d\ln(\phi-\bar{\psi})
-\sum_{m=1}^{k_2}\lim_{D_{\hat{q}_m}\rightarrow
\{\hat{q}_m\}}\oint_{\partial D_{\hat{q}_m}}d\ln(\phi-\bar{\psi})\\
&=2\pi\mathrm{i}(\deg\phi-\deg\psi)~.
\end{split}
\end{equation*}
The last equality follows from Lemma~\ref{lem-pole} (note that
$\psi(q_l)$ is a complex number when $q_l$ is a pole of $\phi$). The conclusion follows.
\end{proof}
\begin{corollary}\label{cor-deg}
Let $\phi$ and $\psi$ be holomorphic maps
from $\overline{M}$ to $\mathbb{C}P^1$ satisfying $\phi\neq\bar{\psi}$ on
$\overline{M}$. Then $\deg\phi=\deg\psi$.
\end{corollary}
\begin{remark}
This corollary could also be proved by considering the mapping $\phi-\bar\psi: \overline{M} \to \mathbb{C}P^1=\mathbb{C}\cup\{\infty\}$
and counting the number of inverse images of $0$ and $\infty$, separately. Proposition~\ref{prop-index}
appears later as a corollary of Theorem~\ref{GB2}.
See \eqref{eq-deg3}.
\end{remark}

\section{Gauss-Bonnet type formulas}

After detailed discussion of singular ends and their indices, now we are in the position to state and prove the following index theorem.
\begin{theorem}\label{GB2}
Let $\vec{x}:M\rightarrow \mathbb{R}^4_1$ be a complete, oriented stationary surface
given by \eqref{x} in terms of $\phi,\psi,\mathrm{d}h$
which are meromorphic functions/forms defined
on a compact Riemann surface $\overline{M}$ (i.e., it is of algebraic type).
It may have good singular ends, but no bad singular ends.
Then the total Gaussian curvature and total normal curvature are related with the indices at the ends $p_j$ (singular or regular) by the following formulas:
\begin{align}
\int_M K^{\perp}\mathrm{d}M &=0~;\label{eq-deg0}\\
\int_M K\mathrm{d}M &=-4\pi \deg\phi+2\pi\sum{_j} \left(\mathrm{ind}^{+}_{p_j}(\phi-\bar{\psi})
+\mathrm{ind}_{p_j}(\phi-\bar{\psi})\right) \label{eq-deg1}\\
&=-4\pi \deg\psi+2\pi\sum{_j} \left(\mathrm{ind}^{+}_{p_j}(\phi-\bar{\psi})
-\mathrm{ind}_{p_j}(\phi-\bar{\psi})\right)~.\label{eq-deg2}
\end{align}
From \eqref{eq-deg1} \eqref{eq-deg2} we have equivalent identities:
\begin{gather}
\sum{_j}\mathrm{ind}_{p_j}(\phi-\bar{\psi})=\deg\phi-\deg\psi~.\label{eq-deg3}\\
\int_M K\mathrm{d}M =-2\pi \left(\deg\phi+ \deg\psi-\sum{_j} \mathrm{ind}^{+}_{p_j}(\phi-\bar{\psi})\right)~.\label{eq-deg4}
\end{gather}
\end{theorem}
\begin{proof}
Without loss of generality, assume that the meromorphic functions $\phi$ on $\overline{M}$ have poles
$\{p_{r+1},\cdots,p_m\}$ which are distinct from the ends $\{p_1,\cdots,p_r\}$. Take disjoint neighborhoods of
$p_1,\cdots,p_r,p_{r+1},\cdots,p_m$ respectively
and denote them by
\[
D_1,\cdots,D_r,D_{r+1},\cdots,D_m.
\]
By assumption, the ends are regular ends or good singular ends,
around which the curvature integral must converge absolutely
(Proposition~\ref{prop-goodsingular}). Hence the limit
\[
\lim_{D_j\to\{p_j\}}\int_{\overline{M}-\bigcup D_j}
(-K+\mathrm{i}K^{\perp})\mathrm{d}M
\]
is well-defined and independent of the limit process $D_j\to\{p_j\}$
for all $1\le j\le m$. Apply Stokes theorem to $\overline{M}-\bigcup_{j=1}^m D_j$
and invoke \eqref{eq-totalcurvature}. We obtain
\begin{align*}
\int_M(-K+\mathrm{i}K^{\perp})\mathrm{d}M
&=2\mathrm{i} \lim_{D_j\to\{p_j\}}\int_{\overline{M}-\bigcup D_j}
\frac{\phi_z\bar{\psi}_{\bar{z}}}{(\phi-\bar{\psi})^2}
\mathrm{d}z\wedge \mathrm{d}\bar{z}
\\
&=2\mathrm{i}\sum_{j=1}^m\lim_{D_j\to\{p_j\}}
\int_{\partial D_j} \frac{\phi_z}{\phi-\bar{\psi}} \mathrm{d}z~.
\end{align*}
The singularities of the 1-form
$\frac{\phi_z}{\phi-\bar{\psi}}\mathrm{d}z$
come from either the zeros of $\phi-\bar\psi$
when $j\le r$, or the poles of $\phi$ when $j\ge r+1$.
In the first case, by Lemma~\ref{lemma-index2} the limit
for each $j=1,\cdots,r$ is either $\mathrm{ind}(\phi-\bar\psi)$ or zero,
depending on whether the index is positive or negative.
In the second case, the limit is the order of the pole of $\phi$
by Lemma~\ref{lem-pole}.
Taking sum we get \eqref{eq-deg0} and \eqref{eq-deg1}.

\eqref{eq-deg2} is derived in a similar fashion by considering $\frac{\bar{\psi}_{\bar{z}}}{\phi-\bar{\psi}} \mathrm{d}\bar{z}$
instead of $\frac{\phi_z}{\phi-\bar{\psi}} \mathrm{d}z$ when using Stokes theorem.
Taking the sum or the difference of \eqref{eq-deg1} and \eqref{eq-deg2}, we obtain \eqref{eq-deg4} and \eqref{eq-deg3}, respectively.
\end{proof}
\begin{remark}
When there are no singular ends, all indices vanish and
we obtain a simplified version of these formulas:
\[
\int_M K\mathrm{d}M =-4\pi \deg\phi=-4\pi \deg\psi.
\]
The reason for $\deg\phi=\deg\psi$ has been explained in the comments
following Corollary~\ref{cor-deg}. In particular, for $M\to \mathbb{R}^3$,
$\phi=-1/\psi$, which is never equal to $\bar\psi$. Hence we get the classical result.
\end{remark}
\begin{remark}\label{rem-normalK}
On the other hand, even in the codimension-2 case, under our hypothesis there is still
\[
\int_M K^{\perp}\mathrm{d}M =0.
\]
This looks peculiar, since one expects to see something
related with the Euler characteristic of the normal bundle.
A natural explanation is as below.

First observe that we may add the \emph{light cone at infinity} to the ambient space $\mathbb{R}^4_1$
and get a conformal compact Lorentz manifold $Q^4_1$ \cite{Ma-Wang}.
This is the same as one did in M\"obius geometry where the
compactification $\mathbb{R}^n\cup\{\infty\}=S^n$ is obtained
via an inversion or the stereographic projection.
Put it rigorously, we may use the classical construction of
\emph{light cone model}. See \cite{Ma-Wang} for details of
Lorentzian conformal geometry, where $Q^4_1$ is identified
with the set of null lines in $\mathbb{R}^6_2$,
which is topologically $S^3\times S^1/\{\pm 1\}$ endowed with
the Lorentz metric $g_{S^3}\oplus (-g_{S^1})$.

Next, a stationary surface satisfying the assumption of
Theorem~\ref{GB2} indeed could be compactified in $Q^4_1$.
This gives a natural realization and visualization of
Huber's compactification $M=\overline{M}\backslash\{p_1\cdots,p_r\}$.
The extended map
\[\overline{M}\to Q^4_1\]
is generally not $C^{\infty}$, but $C^{1,\alpha}$ \cite{Kusner}
(possibly branched at those
$p_j$ located on the light cone at infinity). The non-existence of
\emph{bad singular ends} should guarantee that the normal bundle
has a nice extension to these ends.

Finally, $Q^4_1=S^3\times S^1/\{\pm 1\}$ has a globally defined
time-like vector field. Restricting to any spacelike surface $M\to
Q^4_1$ and considering the projection of this vector field to its
normal plane at each point, we obtain a time-like (non-zero) global
section of the rank-2 normal bundle, which shows that the
normal bundle is trivial.
\end{remark}
\begin{remark}
It is interesting to compare with similar formulas for complete minimal surfaces in $\mathbb{R}^4$ with finite total curvature \cite{Hof-Oss}:
\begin{equation}\label{eq-deg5}
\int_M K\mathrm{d}M =-2\pi(\deg\phi+\deg\psi),~~
\int_M K^{\perp}\mathrm{d}M =2\pi(\deg\phi-\deg\psi).
\end{equation}
Here $\phi,\psi$ correspond to the two components of the
Gauss map into the Grassmannian
\[
\mathrm{G}(2,2)=\{\text{oriented 2-planes in}~\mathbb{R}^4\}=S^2\times S^2.
\]
Note that $\deg\phi,\deg\psi$ are not necessarily equal to each other, and the normal bundle is non-trivial in general. These formulas were obtained earlier by Chern and Spanier \cite{Chern-Spanier} for closed oriented surfaces in $\mathbb{R}^4$.
\end{remark}

\begin{definition}\label{def-multiplicity}
The multiplicity of a regular or singular end $p_j$ for a stationary surface
in $\mathbb{R}^4_1$ is defined to be
\[
\widetilde{d}_j=d_j-\mathrm{ind}^+_{p_j},
\]
where $d_j+1$ is equal to the order of the pole of
the vector-valued meromorphic $1$-form $\vec{x}_z \mathrm{d}z$ at $p_j$.
\end{definition}

At a regular end $p_j$, $\mathrm{ind}^+(p_j)=0$. So our definition of
the multiplicity of an end coincides with the old one for minimal $M^2\to\mathbb{R}^3$.

As a corollary of Theorem~\ref{GB2} and
Definition~\ref{def-multiplicity} we obtain
the following Gauss-Bonnet type theorem.
\begin{theorem}
[the generalized Jorge-Meeks formula]
\label{GB3}
Given an algebraic stationary surface $\vec{x}:M\rightarrow \mathbb{R}^4_1$
whose ends $\{p_1,\cdots,p_r\}=\overline{M}-M$ are regular ends or good singular ends.
Let $g$ be the genus of compact Riemann surface $\overline{M}$,
$r$ the number of ends, and $d_j$ the multiplicity of $p_j$. We have
\begin{equation}\label{eq-jorgemeeks2}
\int_M
K\mathrm{d}M=2\pi\left(2-2g-r-\sum_{j=1}^r \widetilde{d}_j\right)~, ~~\int_M K^{\perp}\mathrm{d}M =0~.
\end{equation}
\end{theorem}
\begin{proof}
Up to a Lorentz rotation \eqref{trans} we may assume that $\phi,\psi$
do not have poles at ends $\{p_1,\cdots,p_r\}$.
Thus these ends are exactly the poles of $\mathrm{d}h$. So
\[
\sum {\rm pole}(\mathrm{d}h)=\sum_{j=1}^r(d_j+1)=
\sum_{j=1}^r\widetilde{d}_j+\sum_{j=1}^r \mathrm{ind}^+ +r.
\]
On the other hand, zeros of $\mathrm{d}h$ should be regular points of
$\vec{x}:M\to \mathbb{R}^4_1$, hence corresponds precisely to
poles of $\phi$ or $\psi$ with the same order. Taking sum we get
\[
\sum {\rm zero}(\mathrm{d}h)=\deg\phi+\deg\psi.
\]
By the well-known formula for the meromorphic $1$-form $\mathrm{d}h$
counting its zeros and poles over a compact Riemann surface $\overline{M}$:
\[
\sum {\rm zero}(\mathrm{d}h)-\sum {\rm pole}(\mathrm{d}h)=-\chi(\overline{M})=2g-2.
\]
Together with $\int_M K\mathrm{d}M =-2\pi \left(\deg\phi+ \deg\psi-
\sum{_j} \mathrm{ind}^{+}\right)$ (\eqref{eq-deg4} in Theorem~\ref{GB2}),
the first formula is proved.
The second one is just \eqref{eq-deg0} in Theorem~\ref{GB2}.
\end{proof}
\begin{remark}
It is interesting to compare with Kusner's version of
Gauss-Bonnet formula involving the total branching order
(Lemma~1.2. in \cite{Kusner}).
Suppose $M\to \mathbb{R}^3$ is a branched immersion with
a $C^{1,\alpha}$ compactification $\overline{M}\to S^3$,
then Kusner says
\[\int_M K \mathrm{d}M=2\pi(\chi(M)-\eta(M)+\beta(M)).\]
In our terms, $\chi(M)=2-2g-r$ is the Euler number of $M$,
$\eta(M)=\sum_{j=1}^r d_j$ is the sum of multiplicities of the ends,
and $\beta(M)$ the total branching order. To relate with our result,
for a good singular end $p$, without loss of generality we may write
a local Weierstrass representation over a punctuated neighborhood of
$z=0$:
\begin{equation}\label{eq-sing-end}
\phi=z^k \phi_0,~\psi=z^{k+l}\psi_0,~\mathrm{d}h=z^{-n}\mathrm{d}z,
\end{equation}
where $n,k,l$ are integers ($l>0$), $\phi_0,\psi_0$ are non-zero
holomorphic functions around $z=0$.

If $k=0, n>0$, we have a regular end where $\vec{x}_z \mathrm{d}z$ has a
$n$-th order pole, hence the multiplicity at this end is $n-1$.

If $k>0, n=0$, we have a good branch point with branching order $k$
which equals $|\mathrm{ind}(\phi-\bar\psi)|=\mathrm{ind}^+(\phi-\bar\psi)$.
Note that $\mathrm{d}s=|\phi-\bar\psi||\mathrm{d}h|=(|z|^k+o(|z|^k))|\mathrm{d}z|$ in this case,
which fails to hold for a singular point/end ($l=0$).

When $n,k,l$ are all positive, it is natural to count both
contributions coming from the poles of $\mathrm{d}h$ and from the zeros
(branching orders) of $\phi-\bar\psi$. This justifies
Definition~\ref{def-multiplicity} and Theorem~\ref{GB3}. (For more
on Gauss-Bonnet theorem involving branch points, see
\cite{Eschenburg} and references therein.)
\end{remark}
\begin{proposition}\label{ineq-multiplicity}
Let $\vec{x}:D^2-\{0\}\rightarrow \mathbb{R}^4_1$ be a regular or a good
singular end which is further assumed to be complete at $z=0$. Then
its multiplicity satisfies $\widetilde{d}\ge 1.$
\end{proposition}
\begin{proof}
We need only to consider a good singular end with a local
Weierstrass representation \eqref{eq-sing-end} and
positive $n,k,l$ as above. By definition, the multiplicity of this end is $d=n-1$,
and the index of singular end is $\mathrm{ind}^+=k$.
The metric $\mathrm{d}s=|\phi-\bar\psi||\mathrm{d}h|=(|z|^{k-n}+o(|z|^{k-n}))|\mathrm{d}z|$
is complete around $z=0$.
This implies $n-k\ge 1$. The period condition \eqref{eq-period}
excludes the possibility of $n-k=1$. Hence $n-k\ge 2$ and we have
$\widetilde{d}=d-k=n-1-k\ge 1. $
\end{proof}
As a direct consequence of this proposition and Theorem~\ref{GB3}, we obtain
\begin{corollary}
[The Chern-Osserman type inequality]
\label{GB2C}
Let $\vec{x}:M\rightarrow \mathbb{R}^4_1$ be an algebraic stationary surface without
bad singular ends, $\overline{M}=M\cup\{q_1,\cdots,q_r\}$, then
\begin{equation}
\int K\mathrm{d}M \leq 2\pi(\chi(M)-r)=4\pi(1-g-r).
\end{equation}
\end{corollary}
\begin{corollary}
[Quantization of total Gaussian curvature]
\label{cor-quantization}
Under the same assumptions of Theorem~\ref{GB3},
when $\phi,\psi$ are not constants (equivalently, when $\vec{x}$ is not a
flat surface in $\mathbb{R}^3_0$), we always have
\[
\int_M K\mathrm{d}M=-4\pi k,
\]
where $k\ge 1$ is a positive integer.
\end{corollary}
\begin{proof}
This follows directly from \eqref{eq-deg1}\eqref{eq-deg2} except the inequality
$\int_M K\mathrm{d}M\le -4\pi$.

Suppose this is not true and $\int_M K\mathrm{d}M\ge 0$ for
a complete stationary surface. By Corollary~\ref{GB2C}, this implies $g+r\le 1$.
There must be at least one end. So $g=0$ and $r=1$. If this unique end is regular,
by \eqref{eq-deg1} we obtain contradiction since $\deg\phi\ge 1$. So this end is a good singular end.
Without loss of generality we may assume that the index at this end is negative.
But again by \eqref{eq-deg1} we have the same contradiction.
\end{proof}

\section{Typical methods to construct new examples}

In this section and below we will construct many new complete algebraic
stationary surfaces $\vec{x}:M\to\mathbb{R}^4_1$ which are topologically
punctuated 2-spheres (genus zero) with finite total curvature.

To find a new stationary surface $\vec{x}:M\to\mathbb{R}^4_1$, the typical method we used is to find out the Weierstrass data according the prescribed information. Specifically, one may choose either of methods below, or combine some of them:

(1) Write out the vector-valued differential $\vec{x}_z \mathrm{d}z$ directly
with prescribed poles or Laurent expansions so that
it has a desired behavior locally (around an end) or globally
(like being a graph over a plane).

(2) Deform the expression of $\vec{x}_z \mathrm{d}z$ or the Weierstrass data
of a known minimal surface in $\mathbb{R}^3$ to obtain new examples.

(3) Using geometric conditions to determine the divisors (the distribution of zeros
and poles) of the Gauss maps $\phi,\psi$ and height differential $\mathrm{d}h$. Then write out the algebraic expressions explicitly and find out the range of the parameters by solving
the regularity condition and the period condition.

\begin{remark}\label{rem-phipsibar}
To construct a regular complete
stationary surface, usually we find out a pair of meromorphic
functions $\phi,\psi$ on a Riemann surface $M$ as candidates of the Gauss maps; then we have to verify that
\[\phi(z)=\bar\psi(z)\]
has no solution $z\in M$. It is quite often that $\phi=\psi(z,t),\psi=\psi(z,t)$ depending on a multi-parameter $t$, and we have to find out the domain $\Omega$ so that when
$t\in \Omega$ the equation $\phi(z,t)=\bar\psi(z,t)$ has no solution $z\in M$.
This type of equation is quite unusual
to the knowledge of the authors. Most of the time we have to
deal with this problem by handwork.
See discussions in the proof to Theorem~\ref{thm-catenoid} and Theorem~\ref{thm-enneper} for example.
Note that $M\to\mathbb{R}^3$ is a rare case where we overcome
this difficulty easily since $\phi=-1/\psi$ will never equal to $\bar\psi$.
\end{remark}

Compared with the case $\mathbb{R}^3$, another difference is that in $\mathbb{R}^4_1$ embedded complete examples are abundant as shown below.

\subsection{Method 1: Prescribing $\vec{x}_z \mathrm{d}z$ with $\langle \vec{x}_z,\vec{x}_z\rangle=0$}

Here we demonstrate this method by constructing stationary graphs over a 2-dimensional plane in $\mathbb{R}^4_1$.

\begin{example}[The graph of a harmonic function]
\label{exa-graph0}
Let $\mathbb{R}^3_0$ denote the 3-dimensional subspace
orthogonal to $(0,0,1,1)$; it is endowed with a degenerate metric. A stationary surface in $\mathbb{R}^3_0$ can always be written as
\begin{equation}
\vec{x}=(u,v,f(u,v),f(u,v)),
\end{equation}
where $f(u,v)$ is a harmonic function of the variables $(u,v)$.
Taking $F(z)=F(u+iv)$ as an entire function on the whole $z$-plane and let $f=\mathrm{Re}(F)$; this will produce a complete graph over
$(u,v)$-plane with flat metric. These are simplest new examples which are completely embedded in $\mathbb{R}^4_1$ with finite total curvature $\int K\mathrm{d}M=\int K^{\perp}\mathrm{d}M=0$.
\end{example}

\begin{example}[A complete graph over $\mathbb{R}^2$]
\label{exa-graph1} Let
\begin{equation}\label{eq-graph1}
\vec{x}_z\mathrm{d}z=\Big(1,\sqrt{2}\mathrm{i},\cosh(z),\sinh(z)\Big)\mathrm{d}z
\end{equation}
which satisfies
\[\langle \vec{x}_z,\vec{x}_z\rangle=0,~~
|\vec{x}_z|^2=3+|\cosh^2(z)|-|\sinh^2(z)|\in[2,4]\subset\mathbb{R}.
\]
It defines a complete stationary graph $\vec{x}:\mathbb{C}\to
\mathbb{R}^4_1$:
\[
\vec{x}(u,v)=2\Big(u,-\sqrt{2}v,\sinh(u)\cos(v),\cosh(u)\cos(v)\Big).
~~~(z=u+\mathrm{i}v,~u,v\in\mathbb{R})
\]
It is singly periodical (with respect to $v$) and not flat.
So the total curvature does not converge absolutely.
Related with this fact, its Weierstrass data is
\[
\phi=(1-\sqrt{2})\mathrm{e}^{-z},~\psi=(1+\sqrt{2})\mathrm{e}^{-z},
~\mathrm{d}h=\frac{1}{2}\mathrm{e}^z.
\]
Each of them has an essential singularity at the end $z=\infty$.
Also note $\phi\ne\bar\psi$.
\end{example}
\begin{remark}\label{rem-embed1}
In contrast with Meeks-Rosenberg's uniqueness theorem of the
helicoid in $\mathbb{R}^3$ \cite{Meeks2}, these examples in
$\mathbb{R}^4_1$ are also complete, simply connected,
and properly embedded. In Proposition~\ref{prop-enneper} we will
see that the Enneper surface could be deformed in $\mathbb{R}^4_1$ to satisfy these conditions. Together with the generalized helicoid in Example~\ref{exa-helicoid} there are lots of such examples.
\end{remark}

To find embedded minimal surfaces in $\mathbb{R}^3$, a basic result \cite{Schoen}
says that any embedded minimal end of finite total curvature must be either a catenoidal end
(asymptotic to a half-catenoid) or a planar end (asymptotic to a plane);
one annular end with finite total curvature is embedded if and only if
the multiplicity is $1$.

In $\mathbb{R}^4_1$ there is much more freedom to construct embedded ends.
The next family of examples shows that the multiplicity of such an end can take the value of any positive integer. In particular these examples show that Schoen's characterization of the catenoid does not generalize to $\mathbb{R}^4_1$ (see the discussion in the introduction).
\begin{example}
[A complete graph over a punctured timelike plane]
\label{exa-graph2}
Write
\begin{equation}\label{eq-graph2}
\vec{x}_z\mathrm{d}z=\left(z^n+\frac{1}{z^n},-\mathrm{i}\big(z^n-\frac{1}{z^n}\big),
\sqrt{3}\mathrm{i},1\right)\mathrm{d}z.~~~n\in\mathbb{Z}, n\ge 2.
\end{equation}
As the previous example, this is embedded as a graph over the $(x_3,x_4)$ plane
(yet punctuated at $(0,0)$) with two ends $z=0,\infty$. Since
\[|\vec{x}_z|^2=2(|z|^n+\frac{1}{|z|^n}+1)\ge 6,\]
this surface is complete and regular. When
$z\to 0$, $(x_1,x_2)\to\infty, (x_3,x_4)\to(0,0)$,
so $z=0$ is a planar end with an asymptotic plane.
$\vec{x}_z\mathrm{d}z$ has a pole of order $n$ at $z=0$, so the multiplicity of this
embedded planar end is $d_0=n-1$.
Similarly the end $z=\infty$ is embedded with multiplicity $d_{\infty}=n+1$. The total Gaussian curvature $-\int K\mathrm{d}M=4\pi n$ is finite due to Theorem~\ref{GB3}.
For the reader's convenience we give
\[
\phi=\bar\lambda\cdot z^n,~\psi=\bar\lambda\cdot\frac{1}{z^n},~
\mathrm{d}h=\lambda \mathrm{d}z,~~~\lambda=\frac{1+\sqrt{3}\mathrm{i}}{2}.
\]
\end{example}

\subsection{Method 2: Deforming known examples}

\begin{example}
[Al\'ias-Palmer deformation]
\label{exa-deform}
Let $\vec{x}:M\to \mathbb{R}^3$ be a minimal surface given by
\[
\vec{x}_z\mathrm{d}z=(1-g^2,\mathrm{i}(1+g^2),2g,0)\omega
\]
and suppose that $\omega, g\omega, g^2\omega$ are all exact 1-forms, i.e., it has no real or imaginary
periods along any closed path. Al\'ias and Palmer \cite{Alias}
introduced a deformation with complex parameter $a\in \mathbb{C}\backslash\mathbb{R_{-}}$:
\[
\vec{x}_z\mathrm{d}z=\Big(1-ag^2,\mathrm{i}(1+ag^2),
(1+a)g,(1-a)g\Big)\omega.
\]
In terms of our Weierstrass data,
\begin{equation}\label{eq-deform}
\phi=\frac{1}{g},~~\psi=-ag=a\cdot\frac{-1}{\phi},
~~\mathrm{d}h=g\omega.
\end{equation}
It is easy to show that when $a$ is not a negative real number, this deformation still satisfies the regularity and period conditions and preserves the completeness property.
\end{example}

By this deformation Al\'ias and Palmer produced a generalization of the Enneper surface in $\mathbb{R}^4_1$ \cite{Alias}.
On the other hand, the catenoid can not be deformed in this way due to the existence of imaginary period.

Below we will see that, given a minimal surface in $\mathbb{R}^3$,
there are many ways to deform it in $\mathbb{R}^4_1$.
In Section~8 we can find more general deformations of the Enneper surface, and show that some of them are embedded.
The generalization of the catenoid in $\mathbb{R}^4_1$ will be
given in Example~\ref{exa-catenoid}.

Here we introduce another typical way of deformation in $\mathbb{R}^4_1$ which yields interesting generalization
of the $k$-noids (which also includes the generalization of the catenoid).

\begin{example}[The generalized Jorge-Meeks $k$-noids]
\label{exa-knoids}

Recall that the classical Jorge-Meeks $k$-noid
$\vec{x}=(x_1,x_2,x_3,0):M\to\mathbb{R}^3\subset \mathbb{R}^4_1$ has genus zero and $k$ catenoidal ends ($k\ge 2$).
It is not embedded when $k\ge 3$, because these catenoidal ends are sure to intersect with each other when they are extended sufficiently far away.
Fix $k$ and let $\lambda=\mathrm{e}^{2\pi\mathrm{i}/k}$.
The $k$-noid is defined by
\begin{equation}\label{eq-knoid}
M=\mathbb{C}\cup\{\infty\}\backslash\{\lambda^j|
,j=1,\cdots,k\},~~ \phi=\frac{-1}{\psi}=
z^{k-1},~~\mathrm{d}h=\frac{z^{k-1}\mathrm{d}z}{(z^k-1)^2}~.
\end{equation}
Observe that it has a $k$-fold rotational symmetry explicitly given by
$z\to \lambda^j z$ on $M$.

Fix two constants $a,b\in \mathbb{C}$ such that
$a^2-b^2=1, |a|^2-|b|^2>0$, and $a,b$
are linearly independent over $\mathbb{R}$ (for example we may take
$a=\frac{\sqrt{3}}{2},b=\frac{\mathrm{i}}{2}$). Deform the corresponding $\vec{x}_z=(v_1,v_2,v_3,0)$ to
\begin{equation}\label{eq-knoid}
(\hat{\vec{x}}_{a,b})_z=(v_1,v_2,av_3,bv_3 )~.
\end{equation}
By integration we obtain the generalized $k$-noid
$\hat{\vec{x}}_{a,b}:M\to\mathbb{R}^4_1$.
\end{example}
The generalized $k$-noid $\hat{\vec{x}}_{a,b}$ is still regular and complete, being conformal (yet not isometric)
to the original $\vec{x}:M\to\mathbb{R}^3$.
The total curvature is the same $\int K\mathrm{d}M=-4\pi(k-1)$.
Because the first two components are the same as the original surface in $\mathbb{R}^3$, and the original $\mathrm{d}h$ is an exact differential, both the horizontal
and vertical period conditions are satisfied.
The difference is that our deformation has no self-intersection.
\begin{proposition}
The generalized $k$-noid $\hat{\vec{x}}_{a,b}$ is embedded in $\mathbb{R}^4_1$.
\end{proposition}
\begin{proof}
Suppose there is $\hat{\vec{x}}_{a,b}(z)=\hat{\vec{x}}_{a,b}(w)$.
Since $v_3\mathrm{d}z=\mathrm{d}h=d\frac{-1}{k(z^k-1)}$ and $a,b$
are linearly independent over $\mathbb{R}$, by comparing the third and fourth
components we know $z^k=w^k$. Thus $w=z\cdot \lambda^j$ with $\lambda=\mathrm{e}^{\frac{2\pi\mathrm{i}}{k}}$.

Now we need only to consider the first two components of $\hat{\vec{x}}_{a,b}(z)$
and $\hat{\vec{x}}_{a,b}(\lambda^j z)$,
which are the same as the first two components of $\vec{x}(z),\vec{x}(\lambda^j z).$
Since $\vec{x}$ has a $k$-fold rotational symmetry on the $(x_1,x_2)$ plane.
Thus $\hat{\vec{x}}_{a,b}(z)=\hat{\vec{x}}_{a,b}(\lambda^j z)$ if, and only if,
their first two components correspond to a fixed point under this rotation.
According to the description of the $k$-noid, this corresponds to
the fixed point of $z\to z\cdot\lambda^j,$ i.e. $z=0$ or $z=\infty$.
But in either case we have $w=z\cdot \lambda^j=z$.
This confirms the embedding property.
\end{proof}

This example, together with the generalized Enneper surfaces (see Proposition~\ref{prop-enneper4}), show that the Lopez-Ros theorem \cite{Ros} does not generalize to $\mathbb{R}^4_1$. (See the discussion in the introduction.)

\begin{remark}
The deformation we introduced above is named the \emph{Lorentz deformation}, and used to construct a generalization of the Chen-Gackstatter surface in one of our preprints \cite{Ma-Xie}.
\end{remark}

\subsection{Method 3: Determine the Weierstrass data with given divisors}

This is the most widely used method in constructing examples,
describing the related moduli space, or showing non-existence results
under given geometric and topological conditions.

For example, this is repeatedly used in our classification of complete stationary surfaces of algebraic type with total curvature $\int K\mathrm{d}M=-4\pi$. See Theorem~\ref{thm-catenoid} and Theorem~\ref{thm-enneper} for the characterizations of the generalized catenoids and the generalized Enneper surfaces, respectively; the proof to our classification theorem is finished in our preprint \cite{Ma-4pi}.

Here we present a by-product of this exploration, that is, an algebraic, complete stationary surface with two good singular ends and smallest possible total Gaussian curvature
$-\int_M K\mathrm{d}M=8\pi$. The derivation of
this example is somewhat technical, and irrelevant to other parts of this paper. So we leave these details to the other paper \cite{Ma-4pi}.

\begin{example}
[Genus zero, two good singular ends and $\int_M K\mathrm{d}M=-8\pi$]
\label{exa-singular1}
\[
M=\mathbb{C}\backslash\{0\},~\phi=z^2(z^2+a),
~\psi=\frac{z^4}{z^2+a},
~\mathrm{d}h=\frac{z^2+a}{z^4}\mathrm{d}z.~~
(a\in\mathbb{C}\backslash\{0\})
\]
This example has the following properties:
\begin{itemize}
\item Its genus $g=0$; the number of ends $r=2$; and $\deg(\psi)=\deg(\psi)=4$.
\item At the end $z=0$,
$\mathrm{ind}_0=2,\tilde{d}_0=1$; at $z=\infty$, $\mathrm{ind}_{\infty}=-2,\tilde{d}_{\infty}=3.$
\item Using either of \eqref{eq-deg1},\eqref{eq-deg2} or \eqref{eq-jorgemeeks2}
we get $\int_M K\mathrm{d}M =-8\pi$.
\end{itemize}
\end{example}

The regularity, completeness and period conditions are easy to verify, except that
we need to find suitable parameter $a$ such that
$\phi\ne\bar\psi$ on $M=\mathbb{C}\backslash\{0\}$. Denote $w=z^2$.
Then $\phi(w)=w(w+a),\psi=\frac{w^2}{w+a}$. When $w\ne0$ we have
\[
\phi(w)=\overline\psi(w) ~~~\Leftrightarrow~~|w+a|^2=w^3/|w|^2.
\]
So $w=r\lambda^j$ with  $\lambda=\mathrm{e}^{2\pi\mathrm{i}/3}, j\in\{0,1,2\}, r\in\mathbb{R}\backslash\{0\}$. Insert this
back into the equality above we get
\[
\phi(w)=\overline\psi(w) ~~~\Leftrightarrow~~
\exists ~j\in\{0,1,2\}, s.t. ~|r\lambda^j+a|^2=r.
\]
It is not difficult to see that when $a$ is a sufficiently large
positive real number (e.g. $a>1$) there is no (positive) real
solution $r$, hence $\phi\ne\bar\psi$ always holds true.
A standard proof is reducing $|r\lambda^j+a|^2=r$ to $r^2-(a+1)r+a^2=0$ when
$j=1,2$, and to $r^2+(2a-1)r+a^2=0$ when $j=0$. Then both
discriminants of these two quadratic equations are negative when
$a>1$.

Such an explicit example is a helpful supplement to the
discussion of good singular ends in Section~4, and to the
Gauss-Bonnet type formulas in Section~5. Below we provide a somewhat different example.
\begin{example}
[Genus zero, one good singular end and $\int K\mathrm{d}M=-8\pi$]
\label{exa-singular2}
\[
M=(\mathbb{C}\backslash\{0\})\cup\{\infty\},
~\phi=z(z^2+az+b),~\psi=\frac{z^2}{z^2+az+b},
~\mathrm{d}h=\frac{z^2+az+b}{z^7}.
\]
The good singular end $z=0$ has $\mathrm{ind}=1$ and $\tilde{d}=5$. We leave
it to the interested reader to verify that it is regular, complete,
and it has no real periods; in particular, the
parameters $a,b\in\mathbb{C}$ could be chosen suitably so that
$\phi\ne\bar\psi$ on $\mathbb{C}\backslash\{0\}$.
\end{example}

\section{The generalized catenoid}

The catenoid in $\mathbb{R}^3$ is not only one of the simplest examples of minimal surfaces, but also plays an important role as a model surface both in the general theory and in many constructions. So it is also important to find a suitable generalization in $\mathbb{R}^4_1$.
At the beginning we tried to find an annulus with two catenoidal ends
(using Method 1 above with Laurent series like \eqref{eq-laurent} \eqref{eq-laurent2} below), and obtained

\begin{example}
[The generalized catenoids]
\label{exa-catenoid}
This is defined over $M=\mathbb{C}\backslash\{0\}$ with
\begin{equation}\label{eq-catenoid}
\phi=z+t,\ \psi=\frac{-1}{z-t},\ \mathrm{d}h=s\frac{z-t}{z^2}\mathrm{d}z.~~~~~
(-1<t<1, s\in\mathbb{R}\backslash\{0\})
\end{equation}
\end{example}
When $t=0$, one obtains again the classical catenoid in $\mathbb{R}^3$.
Conversely, our construction could be viewed as the most natural
deformation of the catenoid preserving the topology, the total curvature and the degree of the Gauss maps.

By contrast, the catenoid in $\mathbb{R}^3_1$ given by
\eqref{eq-catenoid2} has a cone-like singularity, and the Al\'ias-Palmer deformation \cite{Alias} fails to satisfy the period condition.
In view of these facts, the following theorem is a nice
characterization of the generalized catenoids.
\begin{theorem}\label{thm-catenoid}
A completely immersed algebraic stationary surface in $\mathbb{R}^4_1$
with total curvature $\int K=-4\pi$ and two regular ends is
a generalized catenoid given above.
\end{theorem}
\begin{proof}
For a complete and immersed algebraic stationary surface
$M\to\mathbb{R}^4_1$ with $\int K=-4\pi$ and two regular ends, by
Corollary~\eqref{GB2C} it has genus $g=0$. By Huber's theorem $M$ is
conformal to $\mathbb{C}\backslash\{0\}$ with two ends at
$z=0,\infty$.

Next, by \eqref{eq-deg1} and \eqref{eq-deg2}, $\int K=-4\pi$ implies
that the Gauss maps $\phi,\psi$ have degree $1$, hence they are
fractional linear functions on $\mathbb{C}$.

At first sight, $\phi,\psi$ have six coefficients to choose arbitrarily.
But we can apply a Lorentz transformation in the ambient $\mathbb{R}^4_1$
(whose action on $\phi,\bar\psi$ are linear fractional transformations
according to \eqref{trans} in Remark~\ref{rem-trans}),
or a change of complex coordinate $z\to (\alpha z+\beta)/(\gamma z+\delta)$,
to simplify the expressions of $\phi,\psi$.
This is what we want to do below.

Without loss of generality, we suppose $\phi(\infty)=\infty,~\psi(\infty)=0.$
Otherwise, if $\phi(\infty)=-d/c,\psi(\infty)=-\bar{b}/\bar{a}$,
we may use linear fractional transformation
$\phi\to (a\phi+b)/(c\phi+d),
\psi\to  (\bar{a}\psi+\bar{b})/(\bar{c}\psi+\bar{d})$
(which is non-degenerate since $\phi\ne\bar\psi$)
and the effect is as desired.

Up to a change of the complex coordinate $z$ and a Lorentz rotation
using \eqref{trans}, we may normalize the expression of $\phi$ so that
\[\phi=z+t\]
(with parameter $t\in \mathbb{R}$); meanwhile, $\psi=\frac{c}{z+r}$
with parameters $c,r\in\mathbb{C}$.

Once $\phi,\psi$ are given as above with $r\ne 0$,
the height differential $\mathrm{d}h$ must
have a unique simple zero at $z=-r$ and no poles on
$\mathbb{C}\backslash\{0\}$ due to the regularity condition
(Theorem~\ref{thm-period}). Thus
\[
\mathrm{d}h=s\frac{z+r}{z^k}\mathrm{d}z
\]
with parameters $s\in\mathbb{C},k\in \mathbb{Z}$. Completeness implies
$1\le k\le 3$. By Proposition~\ref{ineq-multiplicity},
each end has multiplicity at least $1$, hence $k=2$. When $r=0$ this
expression for $\mathrm{d}h$ is still valid. So we have
\[
\phi=z+t,~~\psi=\frac{c}{z+r},~~ \mathrm{d}h=s\frac{z+r}{z^2}\mathrm{d}z,~~~~t\in
\mathbb{R},~c,r,s\in\mathbb{C}.
\]
To decide the possible values of the parameters, invoking the period
condition \eqref{eq-period}, we deduce
\[s,c\in\mathbb{R}\backslash\{0\},~~~r=-t\in\mathbb{R}.\]
Now that $\phi=z+t,\psi=\frac{c}{z-t}$, by the regularity condition
$\phi\ne\bar\psi$, the parameters $c,t\in\mathbb{R}$ must be chosen
so that the equation
\[
z\bar{z}+t(\bar{z}-z)-c-t^2=0
\]
has no solution $z$. If there is some $z,c,t$ satisfying this equation,
comparing the imaginary and real parts separately shows that either
$t=0,c\ge 0$, or $z\in\mathbb{R}, c+t^2\ge 0$. Thus
$\phi\ne\bar\psi$ always holds true if, and only if, the real
parameters $c,t$ satisfy
\[c<-t^2.\]
After the following change of complex coordinate $z$,
together with a change of frames in ambient space $\mathbb{R}^4_1$ as below (see Remark~\ref{rem-trans}):
\[
\tilde{z}=\frac{1}{\sqrt{-c}}z,~~
\tilde\phi=\frac{1}{\sqrt{-c}}\phi,~~
\tilde\psi=\frac{1}{\sqrt{-c}}\psi,~~ \widetilde{\mathrm{d}h}=\sqrt{-c}\mathrm{d}h~,
\]
we get the desired Weierstrass data
\[
\tilde\phi=\tilde{z}+\tilde{t},~~
\tilde\psi=\frac{-1}{\tilde{z}-\tilde{t}},~~
\widetilde{\mathrm{d}h}=\tilde{s}\frac{\tilde{z}-\tilde{t}}{\tilde{z}^2}d\tilde{z},
\]
with parameters $\tilde{c}=-1,\tilde{s}=\sqrt{-c}s\in\mathbb{R}\backslash\{0\}$ and
\[
\tilde{t}=\frac{t}{\sqrt{-c}}\in(-1,1)\subset\mathbb{R}.
\]
This finishes the proof. In particular, the discussion above has shown that
Example~\ref{exa-catenoid} is completely immersed without singular ends.
\end{proof}

It is interesting to examine the properties of the generalized
catenoid and compare it with the catenoid in $\mathbb{R}^3$.
\begin{proposition}\label{prop-catenoid}
The generalized catenoid has the following properties:

(1) It is embedded in $\mathbb{R}^4_1$.

(2) There is a symmetry between its two ends. In other words,
there is an isometry of $\mathbb{R}^4_1$ which interchanges these two ends
and preserves the whole surface invariant.

(3) Unlike the catenoid in $\mathbb{R}^3$ or $\mathbb{R}^3_1$,
the generalized catenoid has only a finite symmetry group when $t\ne 0$.

(4) Two generalized catenoids are congruent to each other (up to a
dilation and a Lorentz isometry) if and only if their parameters
share the same absolute value $|t|$.

(5) It is not contained in any 3-dimensional subspace when $t\ne 0$.
Each end is embedded and asymptotic to a half catenoid
in a 3-dimensional Euclidean subspace.
\end{proposition}
\begin{proof}
Up to a dilation and a translation we may take $s=1$ and write
\begin{equation}\label{eq-catenoid}
\vec{x}~ =~ 2~\mathrm{Re}\left[\left(z+\frac{t^2+1}{z},
-\mathrm{i}\Big(z+\frac{t^2-1}{z}\Big), 2\ln z,\frac{2t}{z}\right)\right]~.
\end{equation}
To verify embeddedness, it suffices to show that the four components
of $\vec{x}$ determine a unique $z$ on $\mathbb{C}\backslash\{0\}$. The
third component gives the module $|z|$. Combined with this
information, using the fourth and the second component of $\vec{x}$ we can
derive the real and imaginary part of $z$, separately. Thus there is
a unique $z$ corresponding to a given $\vec{x}(z)$. So the generalized
catenoid is embedded.

The end behavior is determined by the Laurent expansion
\begin{equation}\label{eq-laurent}
\vec{x}_z~=~\frac{1}{z^2}\vec{w}_{-2}+\frac{1}{z}\vec{w}_{-1}+\vec{w}_0,
\end{equation}
where the coefficient vectors (written as column vectors) are
\begin{equation}\label{eq-laurent2}
\vec{w}_{-2}=\begin{pmatrix}-(1+t^2)\\ -\mathrm{i}(1-t^2)\\ 0\\ -2t\end{pmatrix},~~
\vec{w}_{-1}=\begin{pmatrix}0\\ 0\\ 2\\ 0\end{pmatrix},~~
\vec{w}_0=\begin{pmatrix}1\\ -\mathrm{i}\\ 0\\ 0\end{pmatrix}~.
\end{equation}
One may verify directly that the Lorentz transformation $\vec{x} \to A_0\vec{x}$
with
\[
A_0=\frac{1}{1-t^2}\begin{pmatrix}
-1-t^2~ &0~~~ &0 &2t\\
0 & 1~~~& 0& 0\\
0 & 0~~~& 1& 0\\
-2t &0~~~ &0 &~1+t^2
\end{pmatrix}
\]
will preserve the vector-valued 1-form $\vec{x}_z\mathrm{d}z$
but interchange its two poles. This proves conclusion (2).

On the other hand, if there is a continuous 1-parameter family of Lorentz transformations
preserving $\vec{x}_z\mathrm{d}z$ as given above, each $A$ in this family must
have invariant subspaces
\begin{align*}
V_{-2}&~=~\mathrm{Span}\{\mathrm{Re}(\vec{w}_{-2}),\mathrm{Im}(\vec{w}_{-2})\}
=\mathrm{Span}\{(\begin{matrix}\frac{1+t^2}{1-t^2}\end{matrix},
0,0,\begin{matrix}\frac{2t}{1-t^2}\end{matrix}),(0,1,0,0)\},\\
V_{-1}&~=~\mathrm{Span}\{\vec{w}_{-1}\}=\mathrm{Span}\{(0,0,1,0)\},\\
V_0~~&~=~\mathrm{Span}\{\mathrm{Re}(\vec{w}_0),\mathrm{Im}(\vec{w}_0)\}
=\mathrm{Span}\{(1,0,0,0),
(0,1,0,0)\}.
\end{align*}
When $t\ne0$, $V_{-2},V_0$ are distinct,
and they together span a Lorentz 3-space in
$\mathbb{R}^4_1$ which is orthogonal to $\vec{w}_{-1}$. From these
information it is easy to see that the Lorentz transformation $A$ is a reflection with respect to the hyperplane orthogonal to $(0,1,0,0)$, or is an identity map when it is orientation-preserving. This proves (3).

We observe that for any spacelike 3-space $V$
in $\mathbb{R}^4_1$, there is a unique timelike vector $\vec{v}$
which is future-oriented (i.e. the fourth component is positive)
and $\langle \vec{v},\vec{v}\rangle=-1, \langle \vec{v},V\rangle=0$.

For $V'=V_{-2}\oplus V_{-1}$, we have
$\vec{v}'=(\frac{2t}{1-t^2},0,0,\frac{1+t^2}{1-t^2})$.

For $V''=V_0\oplus V_{-1}$, we have $\vec{v}''=(0,0,0,1)$.

An algebraic invariant associated with the pair $\{V',V''\}$ is
\[\langle \vec{v}',\vec{v}''\rangle=\frac{1+t^2}{1-t^2},\]
which corresponds to the hyperbolic angle between the two ends,
or between the spacelike 3-spaces $V',V''$
containing the asymptotic half catenoids (see next paragraph).
Thus when the values of $|t|$ are different, the corresponding
generalized catenoids are not equivalent. On the other hand, if we
reverse the sign of $t$, the mapping $x$ will differ by a reflection
according to \eqref{eq-catenoid}. This establishes (4).

By \eqref{eq-catenoid} it is easy to see that $\vec{x}(z)$ is asymptotic
to the catenoid
\[
\vec{x}~ =~ 2~\mathrm{Re}\left(z+\frac{1}{z}, -\mathrm{i}\Big(z-\frac{1}{z}\Big),
2\ln z,0\right)~.
\]
in $V''$ when $z\to \infty, \mathrm{Re}(\frac{1}{z})\to 0$.
At the end $z=0$ one can verify in
a similar way that $x$ is asymptotic to a half catenoid in $V'$
(or by the symmetry between the two ends).
Also note that $V',V''$ span the full $\mathbb{R}^4_1$ when $t\ne 0$.
So the generalized catenoid is not contained in any 3-dimensional subspace.
This finishes the proof to the final conclusion (5).
\end{proof}
\begin{remark}
In the proof to the conclusion (3), we have shown that when $t\ne 0$, the symmetry group of a generalized catenoid has order four; this is generated by the reflection $A_0$, and the reflection
$A_1=\mathrm{diag}(1,-1,1,1)$.
\end{remark}
\begin{remark}\label{rem-rotation}
As to the conclusion (4), we can show that any stationary
surface of revolution (i.e. it has a continuous 1-parameter symmetry group)
must be contained in a 3-dimensional subspace of $\mathbb{R}^4_1$.
Here the proof is omitted.
Regretfully we could not find this result in the literature.
\end{remark}

As in $\mathbb{R}^3$, for a stationary surface $\vec{x}:M\to\mathbb{R}^4_1$
there is an associated family of stationary surfaces
$\vec{x}_{\theta}:M\to\mathbb{R}^4_1$ with $(\vec{x}_{\theta})_z \mathrm{d}z=\mathrm{e}^{\mathrm{i}\theta}\vec{x}_z \mathrm{d}z$.
They are locally isometric to each other, yet the period condition might be violated and the topological type might be different. A typical example is the associated family of the generalized catenoid, which helps us to define the generalized helicoid.
\begin{example}
[The associated family of the generalized catenoid]
\label{exa-associate}
This is represented on $\mathbb{C}\backslash\{0\}$ using \eqref{x} with
\begin{equation}\label{eq-associate}
\phi=z+t,\ \psi=\frac{-1}{z-t},\ \mathrm{d}h=\lambda\frac{z-t}{z^2}\mathrm{d}z,
\end{equation}
where the parameter $\lambda$ is complex and $t\in (-1,1)$ is real as in Example~\ref{exa-catenoid}.
When $\lambda$ is not a real number, the period condition is not satisfied,
and the corresponding stationary surface in $\mathbb{R}^4_1$
is conformal to the covering space $\mathbb{C}$.
\end{example}
\begin{example}
[The generalized helicoid]
\label{exa-helicoid}
When $\lambda$ is purely imaginary in \eqref{eq-associate}, we get
the generalized helicoid in $\mathbb{R}^4_1$.
\end{example}
It is well-known that the classical catenoid and helicoid in
$\mathbb{R}^3$ are embedded, but any other surface in their
associated family is not. The same is true in $\mathbb{R}^4_1$.
\begin{proposition}\label{prop-helicoid}
The generalized catenoids and the generalized helicoids are embedded. Any other
stationary surface in the associated family \eqref{eq-associate}
has self-intersection points.
(Indeed, on the universal covering we have a simply-connected
stationary surface whose unique end is not an embedded end.)
\end{proposition}
\begin{proof}
Given $\phi,\psi,\mathrm{d}h$ as in \eqref{eq-associate},
we write out the immersion explicitly:
\begin{equation}\label{eq-helicoid}
\vec{x}~ =~ 2~\mathrm{Re}\left[\lambda\left(z+\frac{t^2+1}{z},
-\mathrm{i}\Big(z+\frac{t^2-1}{z}\Big),
2\ln z,\frac{2t}{z}\right)\right].
\end{equation}
Note that $\ln z$ is a multi-valued function on $\mathbb{C}\backslash\{0\}$
whose imaginary part is given by the argument of $z$.
There are three cases to consider:

{\bf Case 1:} $\lambda=a\in \mathbb{R}$ is real and non-zero.
This time we get the generalized catenoid which is shown to be embedded
in Proposition~\ref{prop-catenoid}.

{\bf Case 2:} $\lambda=\mathrm{i}b$ is purely imaginary with $b\in \mathbb{R}$.

Notice that $\mathrm{Re}(\mathrm{i}b\ln z)$ (the third component)
determines the argument of $z$.
In other words, if $\vec{x}(z)$ and $\vec{x}(z')$ share the same third component, then
$z,z'\in\mathbb{C}\backslash\{0\}$ must be located
in the same ray emanating from $z=0$.
The fourth component of $\vec{x}$ is $\mathrm{Re}(\mathrm{i}bt/z)$, from which we
can fix the module $|z|$ when $z$ is not real.
Thus the imaginary and the real part of
the corresponding $w=\ln z\in \mathbb{C}$ on the universal covering
are both determined uniquely.

In the exceptional case when $z$ is on the real axis, because the second
component of $\vec{x}(z)$ is essentially $z-\frac{1-t^2}{z}$,
a monotonic function of the real variable $z$ when $t\in (-1,1)$,
we still conclude that $z\to \vec{x}(z)\in \mathbb{R}^4_1$ is injective.
This shows that a generalized helicoid is embedded.

{\bf Case 3:} $\lambda=a+\mathrm{i}b$ and $a,b\in \mathbb{R}$ are both non-zero.

Denote $z=\rho(\cos\theta+\mathrm{i}\sin\theta)\in\mathbb{C}\backslash\{0\}$,
and $w=\ln z=\ln\rho+ \mathrm{i}\theta\in\mathbb{C}$ is a lift to the universal covering space $\mathbb{C}$ on which the mapping $\vec{x}$
is single-valued. We are looking for self-intersection points, i.e., point
pairs $\{z,z'\}$ such that $\ln z\ne \ln z', \vec{x}(z)=\vec{x}(z')$.

Given the expression \eqref{eq-helicoid}, we may obtain the values
of $\mathrm{Re}(\lambda z),\mathrm{Re}(\frac{\lambda}{z})$ from the
first and the fourth components of $\vec{x}(z)$, and the value of
$\mathrm{Im}(\lambda z+\frac{\lambda(t^2-1)}{z})$ from the second
component. Combining these information together, we can derive the
value of $z+\frac{t^2-1}{z}$ from $\vec{x}(z)$. Thus $\vec{x}(z)=\vec{x}(z')$ implies
$z'=z$ or $z'=\frac{t^2-1}{z}$. The case $z'=z$ whose lifts differ
by $w'-w=2n\pi$ will not give self-intersection points (by comparing
the values of the third component of $\vec{x}(z)$). As a consequence, we have a self-intersection
if, and only if, the following equations are satisfied
simultaneously:
\begin{gather}
z'=\frac{t^2-1}{z},\label{eq-intersect1}\\
\mathrm{Re}(\lambda z)=\mathrm{Re}(\lambda z'),\label{eq-intersect2}\\
\mathrm{Re}(\lambda \ln z)=\mathrm{Re}(\lambda \ln z').
\label{eq-intersect3}
\end{gather}
Denote
\[
\lambda=a+\mathrm{i}b, ~\ln z=\ln\rho+ \mathrm{i}\theta,
~\ln z'=\ln(\frac{t^2-1}{z})=\ln(\frac{1-t^2}{\rho})+\mathrm{i}(2n+1)\pi
-\mathrm{i}\theta),
\]
where $n\in\mathbb{Z}$ is an arbitrary integer. We recognize that
the solutions $z=\rho \mathrm{e}^{\mathrm{i}\theta}$ are
located on a family of logarithmic spirals
\begin{equation}\label{eq-intersect4}
L_n:~2a\ln\rho-2b\theta=c_n
\end{equation}
where $c_n=a\ln(1-t^2)-b(2n+1)\pi$ is a constant depending only on
the values of the parameter $a,b,t,n$.

On the other hand, \eqref{eq-intersect1} and \eqref{eq-intersect2} imply
\begin{equation}\label{eq-intersect5}
\frac{a\cos\theta}{b\sin\theta}=\frac{\rho^2+t^2-1}{\rho^2-t^2+1}.
\end{equation}
Consider the limit process $z\to \infty$ along the logarithmic spiral $L_n$. The solution $\theta=\theta(\rho)$ of \eqref{eq-intersect4} is monotonically increasing when $\rho\to
+\infty$, and the range is $\theta(\rho)\in(-\infty,+\infty)$.
Then we insert this pair $(\rho,\theta(\rho))$ into \eqref{eq-intersect5}. The right hand side of \eqref{eq-intersect5} tends to
$1$, and the left hand side oscillates between $(-\infty,+\infty)$ periodically.
Thus there are infinitely many $(\rho,\theta)$
satisfying \eqref{eq-intersect4} and \eqref{eq-intersect5} simultaneously,
which finishes our proof.
\end{proof}

\section{The generalized Enneper surfaces}

\begin{example}[The generalized Enneper surface]
\label{exa-enneper} This is given by
\begin{equation}\label{eq-enneper1}
\phi=z,\ \psi=\frac{c}{z}~,\ \mathrm{d}h=s\cdot z\mathrm{d}z,\
\end{equation}
or
\begin{equation}\label{eq-enneper2}
\phi=z+1,\ \psi=\frac{c}{z}~,\ \mathrm{d}h=s\cdot z\mathrm{d}z,
\end{equation}
with complex parameters $c,s\in\mathbb{C}\backslash\{0\}$.
It is clear that this family of examples are complete without any periods.
It has no singular points if and only if the parameter $c=c_1+\mathrm{i}c_2$
is not zero or positive real numbers in the first case \eqref{eq-enneper1},
and
\begin{equation}\label{eq-enneper3}
c_1-c_2^2+\frac{1}{4}<0
\end{equation}
in the second case \eqref{eq-enneper2}.

In \eqref{eq-enneper1}, when $c=-1$ we obtain
the classical Enneper surface in $\mathbb{R}^3$; otherwise it is almost the same as the Al\'ias-Palmer deformation in
Example~\ref{exa-deform}.
\end{example}
\begin{theorem}\label{thm-enneper}
A completely immersed algebraic stationary surface in
$\mathbb{R}^4_1$ with $\int K=-4\pi$ and one regular
end is a generalized Enneper surface given above.
\end{theorem}
\begin{proof}
By \eqref{eq-deg1} \eqref{eq-deg2} and the assumption $\int K=-4\pi$,
$\phi,\psi$ must be meromorphic functions of degree $1$. Since on
tori and higher genus compact Riemann surfaces there are no such
meromorphic functions, the genus must be $0$
and $\phi,\psi$ are linear fractional functions. Suppose
$M=\mathbb{C}$ with an end at $z=\infty$.  Then Theorem~\ref{GB3}
implies that the end has multiplicity $3$.

As in the proof to Theorem~\ref{thm-catenoid}, without loss of
generality we may suppose $\phi(\infty)=\infty,~\psi(\infty)=0$ (up
to a Lorentz rotation in $\mathbb{R}^4_1$).

To further simplify the expressions of linear fractional functions
$\phi,\psi$, let us consider the unique pole of $\psi$ and the unique zero of $\phi$.
If these two points coincide, we may suppose this is the point $z=0$
(up to a linear fractional transformation of the coordinate $z$).
Then it is easy to see that we can simplify to get \eqref{eq-enneper1}.

Otherwise, suppose the pole of $\psi$ is $z=0$ and
the zero of $\phi$ is $z=-1$. (Using linear fractional transform
to change the complex coordinate $z$, we
may map three given points to $\infty,0,-1$ on the Riemann sphere.)
Then we have
\[
\phi=a(z+1),\ \psi=\frac{c}{z},\ \mathrm{d}h=s\cdot z\mathrm{d}z,\
a,c,s\in\mathbb{C}\backslash\{0\}.
\]
The height differential $\mathrm{d}h$ must have this form due to the
regularity condition (Theorem~\ref{thm-period}). Then using \eqref{trans}
\[
\phi\to \phi/a, ~\psi\to\psi/\bar{a},~ \mathrm{d}h\to |a|\cdot \mathrm{d}h,
\]
we get the desired Weierstrass data in \eqref{eq-enneper2}.

The period condition \eqref{eq-period} is obviously satisfied. By
the regularity condition $\phi\ne\bar\psi$, the parameters
$c=c_1+\mathrm{i}c_2$ must be chosen so that the equation
\[
0=z\bar{z}+\bar{z}-\bar{c}=u^2+v^2+u-\mathrm{i}v-c_1+\mathrm{i}c_2
\]
about $z=u+\mathrm{i}v$ has no solution. If there are some $z,c$ satisfying
this equation, comparing the imaginary and real parts separately
shows that
\[v=c_2,~~ u^2+u+c_2^2-c_1=0.\]
Thus $\phi\ne\bar\psi$ for any $w\in\mathbb{C}\cup\{\infty\}$ if, and only if, the discriminant of the quadratic
equation above (about $u$) is negative, i.e.
$c_1-c_2^2+\frac{1}{4}<0.$ This is exactly \eqref{eq-enneper3}.
\end{proof}

\begin{remark}
We have classified complete algebraic stationary surfaces in $\mathbb{R}^4_1$ with total curvature $-\int K \mathrm{d}M=4\pi$
\cite{Ma-4pi} in another paper. Such surfaces are exactly the generalized catenoids and Enneper surfaces. This generalizes
Osserman's classification result in $\mathbb{R}^3$.
\end{remark}

The Enneper surface in $\mathbb{R}^3$ always
intersect with itself along a curve, hence the Enneper end is not embedded.
In contrast, our deformation in $\mathbb{R}^4_1$ in general
has an embedded end (as demonstrated in Proposition~\ref{prop-enneper4}),
or could be embedded globally in the following situation.
\begin{proposition}\label{prop-enneper}
The generalized Enneper surface in \eqref{eq-enneper2} is embedded
when the parameter $c<-\frac{1}{4}$ is real and $s$ is not real or purely imaginary.
\end{proposition}
\begin{proof}
Using \eqref{eq-enneper2}, we write out a generalized Enneper
surface explicitly:
\[
\vec{x}=\mathrm{Re}\left[s\left( \frac{z^3}{3}+\frac{z^2}{2}+cz,
-\mathrm{i}\Big(\frac{z^3}{3}+\frac{z^2}{2}-cz\Big),
\frac{1-c}{2}z^2-cz,\frac{1+c}{2}z^2+cz\right)\right]~.
\]
If there is an intersection with $\vec{x}(z_1)=\vec{x}(z_2)$ for some $z_1\ne
z_2$, by the expression above, any of the following three functions
must take the same value at $z_1$ and $z_2$:
\begin{equation}\label{eq-embed4}
\mathrm{Re}(sz^2),~\mathrm{Re}(cs\frac{z^2}{2}-csz),~
\frac{sz^3}{3}+\frac{sz^2}{2}+\bar{s}\bar{c}\bar{z}.
\end{equation}
When $c\in\mathbb{R}\backslash\{0\}$, from the first two functions
above we deduce that $\mathrm{Re}(sz)$ also takes the same value at
$z_1$ and $z_2$.

Without loss of generality, suppose $s=\mathrm{e}^{\mathrm{i}\theta}$. By assumption
$\sin\theta\ne 0,\cos\theta\ne 0$. Denote
\[w=sz=u+\mathrm{i}v,~~~~u,v\in\mathbb{R}.\]
Inserting this into functions in \eqref{eq-embed4} and $\mathrm{Re}(sz)$, we see that each of the following three functions of $w$ must take the same value
at $w_1=u_1+\mathrm{i}v_1$ and $w_2=u_2+\mathrm{i}v_2$:
\begin{equation}\label{eq-embed3}
\mathrm{Re}(w); ~~\mathrm{Re}(\bar{s}w^2); ~~
\mathrm{Re}\left(\frac{w^3}{3}+\frac{sw^2}{2}+cs^2\bar{w}\right).
\end{equation}
By the first one, $u_1=u_2$. Combined with the second one, we obtain
a function of $v$ which takes the same value for $v=v_1,v_2$ (since $w_1\ne w_2$, we know $v_1\ne v_2$) as below:
\[
\cos\theta \cdot v^2-\sin\theta\cdot 2u_1v.
\]
By our assumptions, $s$ is not purely imaginary, so $\cos\theta\ne 0$, and this is a quadratic function. Vieta's formula implies
\begin{equation}\label{eq-embed1}
v_1+v_2=2u_1\tan\theta.
\end{equation}
Similarly, by the third function in \eqref{eq-embed3} above and
$u_1=u_2$, we know the quadratic function of $v$
\[
-(u_1+\frac{\cos\theta}{2})v^2+(c\sin 2\theta -u_1\sin\theta )v
\]
takes the same value at $v=v_1,v_2$. The coefficient
$u_1+\frac{\cos\theta}{2}$ is non-zero. (Otherwise, if  $u_1+\frac{\cos\theta}{2}=0=c\sin 2\theta
-u_1\sin\theta$, we deduce $c=\frac{-1}{4}$; if
$u_1+\frac{\cos\theta}{2}= 0\ne c\sin 2\theta -u_1\sin\theta$, there must be $v_1=v_2$. In both cases we find contradiction.) Again by
Vieta's formula,
\begin{equation}\label{eq-embed2}
v_1+v_2=\frac{c\sin 2\theta
-u_1\sin\theta}{u_1+\frac{\cos\theta}{2}}.
\end{equation}
Combining \eqref{eq-embed1} with \eqref{eq-embed2}, we find that
$u_1\in\mathbb{R}$ satisfies
\[
u_1^2+\cos\theta u_1-c\cos^2\theta=0.
\]
Thus the discriminant of this quadratic function of $u_1$ is
non-negative. This is a contradiction with our assumption that $c<\frac{-1}{4}$ and $\cos\theta\ne 0$.
This finishes our proof.
\end{proof}
\begin{remark}
In general, embeddedness and transversal self-intersection are both
open properties with respect to the parameter $s$. Thus it is
interesting to know what will happen when $s$ tends to a real number
for the generalized Enneper surfaces as in
Proposition~\ref{prop-enneper}. One can verify that $\vec{x}$ intersect
itself along a curve when $s$ is a non-zero real number, which is
consistent with the open property just mentioned, because such an
intersection pattern is not transversal.
\end{remark}

For \eqref{eq-enneper2} with other values of parameter $c,s$, in
general it is hard to determine whether the corresponding Enneper
surfaces are embedded or not. But in the special case of
\eqref{eq-enneper1}, we have a clear conclusion. We may state this
result for a class of similar examples as below:
\begin{proposition}\label{prop-enneper4}
The simply connected Enneper-type surface $\vec{x}:\mathbb{C}\to
\mathbb{R}^4_1$ with
\begin{equation}\label{eq-enneper4}
\phi=z^k,\ \psi=\frac{c}{z^k}~,\ \mathrm{d}h=s\cdot z^k\mathrm{d}z, ~~
k\in\mathbb{Z}^+, ~c,s\in\mathbb{C}\backslash\{0\}
\end{equation}
is regular with two self-intersection points when
$c\notin\mathbb{R}$. Thus outside a compact subset it has an
embedded end (of multiplicity $d=2k+1$).
\end{proposition}
\begin{proof}
The regularity and completeness is easy to show. In particular, we
know $\phi\ne\bar\psi$ for $z\in\mathbb{C}$ (otherwise from
$\phi=z^k,\psi=c/z^k$ we deduce $c\in \mathbb{R}$, a contradiction),
and the end at $z=\infty$ is not a singular end.

By \eqref{eq-enneper4} we obtain
\[
\vec{x}=\mathrm{Re}\left[s\left( \frac{z^{2k+1}}{2k+1}+cz,
-\mathrm{i}\Big(\frac{z^{2k+1}}{2k+1}-cz\Big),
\frac{1-c}{k+1}z^{k+1},\frac{1+c}{k+1}z^{k+1}\right)\right]~.
\]
We compare the third and the fourth components of $\vec{x}$, which are
both linear combination of the real and imaginary parts of
$z^{k+1}$. Since $c\notin\mathbb{R}$, these two combinations are
linearly independent over $\mathbb{R}$. Thus when there is a
self-intersection with $\vec{x}(z)=\vec{x}(w), z\ne w$, we conclude
\[
z^{k+1}=w^{k+1},~~w=z \cdot \mathrm{e}^{\frac{2\pi\mathrm{i}}{k+1}}.
\]
The first and second components of $\vec{x}(z),\vec{x}(w)$ are equal, hence
$f(z)=f(w)$ where
\[
f(z)=\frac{s}{2k+1}z^{2k+1}+\bar{s}\bar{c}\bar{z}.
\]
So $ f(z)=f(w)=f(z \mathrm{e}^{\frac{2\pi\mathrm{i}}{k+1}})=\mathrm{e}^{\frac{-2\pi\mathrm{i}}{k+1}}f(z)$, and
\[
f(z)=0,~~\text{i.e.}~\frac{s}{2k+1}z^{2k+1}=-\bar{s}\bar{c}\bar{z}.
\]
There are $2k+2$ solutions to this equation, which have the same
module $\sqrt[2k]{(2k+1)|c|}$ (a fixed constant), and differ with
each other by a factor $\lambda^j (0\le j\le 2k+1)$ with
$\lambda=\mathrm{e}^{\frac{\pi \sqrt{-1}}{k+1}}$. Given one solution $z_0$,
then $\{z_0,z_0\lambda^2,\cdots,z_0\lambda^{2k}\}$ are mapped by $\vec{x}$
to the same point in $\mathbb{R}^4_1$, which is a self-intersection
of multiplicity $k+1$. For
$\{z_0\lambda,z_0\lambda^3,\cdots,z_0\lambda^{2k+1}\}$ we get the
other self-intersection. This finishes our proof.
\end{proof}
\begin{remark}
The example \eqref{eq-enneper4} can be viewed as a generalization of
both the example \eqref{eq-enneper1} and the Enneper surfaces in
$\mathbb{R}^3$ with higher dihedral symmetry.
\end{remark}

\section{Open problems}

The previous results on finite total Gaussian curvature and on embedded examples
motivate us to consider some deeper and harder problems.

\begin{problem}\label{prob-essential}
Introduce suitable index for a stationary end
$\vec{x}:D\backslash\{0\}\to\mathbb{R}^4_1$ whose Gauss map $\phi$ or
$\psi$ has an essential singularity at $z=0$ and the integral of
Gaussian curvature converges absolutely around this end.
(Then we shall establish a Gauss-Bonnet type theorem for complete stationary surfaces
which must involve such indices and the total Gaussian curvature.)
\end{problem}

We have constructed examples like \eqref{exa-essen} with finite
total Gaussian curvature and essential singularities for the Gauss
maps $\phi,\psi$. It is surprising that the total Gaussian curvature
$\int K\mathrm{d}M$ (as well as $\int K^{\perp} \mathrm{d}M$) are still quantized.

Indeed, according to \cite{Hartman}, under the assumption of finite
total curvature, the area of geodesic balls of radius $r$ at any
fixed point must grow at most quadratically in $r$. Moreover,
\begin{equation}\label{eq-hartman}
\int_M K \mathrm{d}M +\lim_{r\to \infty}\frac{2A(r)}{r^2}=2\pi \chi(M).
\end{equation}
This formula unifies Jorge-Meeks formula \eqref{eq-jorgemeeks} (see
also Theorem~\ref{GB3}) and other Gauss-Bonnet type formulas.

In our opinion, it looks plausible to introduce some topological
index for a wide class of essential singularities, which we desire
to be simple to compute and coincide with the area growth rate above
when the end has finite total curvature. Once this is done we obtain
Gauss-Bonnet type theorem by \eqref{eq-hartman}. Regretfully it is
unclear whether finite total curvature can determine the types of
essential singularities of the Gauss maps $\phi,\psi$ at one end in
a satisfying way.

\begin{problem}\label{prob-collin}
Can we extend Collin's theorem to $\mathbb{R}^4_1$? In other words,
assume that $\vec{x}:M\to\mathbb{R}^4_1$ is a properly embedded complete
stationary surface of finite topological type (i.e., $M$ is
homeomorphic to a compact surface with finite punctures), and the
number of ends is at least two, does $\vec{x}$ always has finite total
Gaussian curvature?
\end{problem}

In Section~6 to 8 we have constructed many embedded complete
stationary surfaces in $\mathbb{R}^4_1$, and refuted the conclusions
of the Lopez-Ros theorem and Meeks-Rosenberg theorem under the
assumption of embeddedness (see the introduction and Section~6).
A deeper result on properly embedded minimal surfaces in
$\mathbb{R}^3$ is

\medskip
\textbf{Collin's theorem~\cite{Collin}}~~
If $M\subset\mathbb{R}^3$ is a properly embedded minimal surface
with finite topology and more than one end, then $M$ has finite
total Gaussian curvature.

\medskip
It seems that embeddedness is not quite restrictive for complete
stationary surfaces in the 4-dimensional Lorentz space according to our observations before. So at the beginning we tried to find simple counterexamples to the conclusion of Collin's theorem in
$\mathbb{R}^4_1$.

The first attempt is by using the
stationary graph like Example~\ref{exa-graph1}, which is complete,
embedded, with infinite total curvature. If we can modify \eqref{eq-graph1} by introducing one more end while preserving
other proterties, then we get a counterexample. So we consider
\[
\vec{x}_z\mathrm{d}z=(1,\sqrt{2}\mathrm{i},
\cosh(z^2+z^{-2}),\sinh(z^2+z^{-2}))\mathrm{d}z.
\]
Obviously, $\vec{x}_z$ is still isotropic, and the term $z^2+z^{-2}$ is to
introduce a new pole at $z=0$ without any periods.
By the first two components, $\vec{x}=\mathrm{Re}\int \vec{x}_z \mathrm{d}z$ is still a graph and embedded in $\mathbb{R}^4_1$.
It is also not difficult to compute out the Gauss maps
$\phi,\psi$ explicitly and then to show the total absolute curvature diverges at $z=\infty$. Thus it meets our expectation perfectly. Unfortunately, it fails to be complete at $z=0$. So this is not a counterexample.

The second candidate modifies the construction above in another way:
\[
\vec{x}_z\mathrm{d}z=\left(z^k-\frac{1}{z^k},
-\mathrm{i}\big(z^k+\frac{1}{z^k}\big),
2\cosh(z),2\sinh(z)\right)\mathrm{d}z.
\]
The Gauss maps are $\phi=z^k \mathrm{e}^z,~\psi=-\frac{\mathrm{e}^z}{z^k}$,
which is the same as Example~\ref{exa-essen2} and \ref{exa-essen3}. So it has finite total curvature when $k\ge 2$.
If $k=1$, although the total curvature diverges as desired, the period condition is not satisfied and the universal covering surface is $\mathbb{C}$ with a unique end. In either case it is not a counterexample.

The third candidate is the generalized catenoid in
Example~\ref{exa-catenoid} which is complete, embedded, with genus
zero and two ends; but now we take $t=1$, i.e.
\[
\phi=z+1,\ \psi=\frac{-1}{z-1},\ \mathrm{d}h=\frac{z-1}{z^2}\mathrm{d}z.
\]
This is still regular and embedded as shown in
Proposition~\ref{prop-catenoid}. Since now we have a
bad singular end $z=0$, the total Gaussian curvature does not
converge absolutely (see Section~4). The trouble is that
it is not complete at the singular end $z=0$ as required
(because $\mathrm{d}s=|\vec{x}_z||\mathrm{d}z|=|\phi-\bar\psi||\mathrm{d}h|
=|\frac{z\bar{z}+\bar{z}-z}{z\bar{z}}||\mathrm{d}z|$, and the integration of $|\vec{x}_z||\mathrm{d}z|$ along the real axis
$z\in\mathbb{R}$ gives a finite number).

Because simple counter-examples are not found, we suspect that
Collin's theorem might still be true for
stationary surfaces in $\mathbb{R}^4_1$. This question might deeply influence our understanding of
the embedding problem as well as the surface class with finite total curvature.

\begin{problem}\label{prob-normalK}
Assume that a complete stationary surface in $\mathbb{R}^4_1$ has
finite total Gaussian curvature, i.e. $\int |K| \mathrm{d}M<+\infty$.
Does it imply that $\int |-K+\mathrm{i}K^\bot| \mathrm{d}M<+\infty$?
\end{problem}

This problem was mentioned in Remark~\ref{rem-totalK}. In particular,
if the conclusion is true, we conjecture that
the integration of the normal curvature vanishes, i.e., $\int K^\bot \mathrm{d}M=0.$
See Remark~\ref{rem-normalK} for related discussions.

\begin{problem}\label{prob-fujimoto}
Estimate the possible number of exceptional values of the Gauss maps $\phi$
and $\psi$ of complete stationary surfaces in $\mathbb{R}^4_1$. In
particular, can we generalize Fujimoto's Theorem \cite{Fujimoto}to
$\mathbb{R}^4_1$?
\end{problem}

When the complete stationary surface is assumed to be algebraic, we
can use the same method as Osserman \cite{Osser} to show that if the
surface has no singular ends, then the exceptional values of either of
$\phi,\psi$ is no more than $4$ \cite{Ma-exceptional}. If we
allow singular ends, then we can obtain a weaker conclusion that at
least one of the exceptional values of any of $\phi$ and $\psi$ is
no more than $4$. If we add the condition of finite total Gaussian
curvature, then the number $4$ could be replaced by $3$. All these
are similar to the case of $\mathbb{R}^3$. The next step is to
assume only completeness and try to generalize Fujimoto's Theorem
\cite{Fujimoto}.

Note that it is also interesting to consider the upper bound of the
exceptional spacelike or timelike directions \cite{Ma-exceptional}.

\begin{problem}
Apply our theory to study spacelike Willmore surfaces in Lorentzian
space forms and Laguerre minimal surfaces in $\mathbb{R}^3$.
\end{problem}

As mentioned in the introduction, this is the original motivation
for our exploration reported in this paper. Stationary surfaces in
$\mathbb{R}^4_1$ are special examples of spacelike Willmore surfaces
in Lorentzian space forms \cite{Ma-Wang}. When being complete with
planar ends (the definition is similar to the case of
$\mathbb{R}^3$), such surfaces will compactify to be compact
Willmore surfaces in $\mathbb{Q}^4_1$, the universal
compactification of 4-dimensional Lorentz space forms. This
construction yields all compact Willmore 2-spheres in
$\mathbb{Q}^4_1$ \cite{Ma-Wang}, \cite{WangP}. In particular, it is
interesting to know whether there exist Willmore 2-spheres in
$\mathbb{Q}^4_1$ with Willmore functional $4\pi k, (k=2,3,5,7)$.
Note that $k=2,3,5,7$ are all exceptional values for immersed Willmore
2-spheres in $\mathbb{S}^3$ \cite{Bryant84}, \cite{Bryant88}, \cite{Heller}.

A stationary surface in $\mathbb{R}^4_1$ also corresponds to the
so-called Laguerre Gauss map of a Laguerre minimal surface in
$\mathbb{R}^3$ \cite{CPWang2008}. So our theory provide a direct
method to construct examples of Laguerre minimal surfaces and to
study their global geometry.

\begin{problem}\label{prob-phipsi}
Can we have any general result about the equation
\[\phi=\bar\psi~~~~~~~~~~(*)\]
for a pair of meromorphic functions $\phi,\psi$ over a Riemann surface?
\end{problem}

To the best of our knowledge, there is no general theory on such equations.
We collect some known facts on this problem.
\begin{itemize}
\item  There exist many pairs of functions $\{\phi,\psi\}$
for which there is no solutions to equation $(*)$.
For example, in Example~\ref{exa-deform},~\ref{exa-catenoid} and
\ref{exa-enneper} we take $\psi\equiv\frac{c}{\phi+a}$.
When the parameter $a,c$ are chosen suitably there is no solution to
$\frac{c}{\phi+a}=\psi=\bar\phi$. Another example is Example~\ref{exa-essen}.

\item  For meromorphic functions $\phi,\psi$ over compact Riemann surface $M$,
if there is no solution to equation $(*)$, then $\deg\phi=\deg\psi$ (see Corollary~\ref{cor-deg} or Theorem~\ref{GB2}). This a necessary but not sufficient condition for the non-existence of solutions.

\item Regard $(*)$ as zeros of complex harmonic function $\phi-\bar\psi$.
Since this is a complex-valued function with convergent power series
(analytic function), its zero locus is a union of isolated points and
some analytic arcs which might meet at some vertex (at each vertex there are
finitely many of such arcs).
\end{itemize}

In two concrete problems we need to deal with this technical trouble. One is in the classification of complete algebraic stationary surfaces with total curvature $-\int K \mathrm{d}M=4\pi$ \cite{Ma-4pi}, where we rule out the
case with good singular end by showing that there always exist other singular points with $\phi=\bar\psi$.

The other one is in Problem~1. When considering one end with an essential singularity of
given $\phi, \psi$ , we want to know whether
there will be infinitely many singular points $z$ such that $\phi(z)=\overline{\psi(z)}$
in the neighborhood of this end. By the theorem of Weierstrass,
around an essential singularity of a holomorphic function $\phi$,
it may take almost every complex value infinitely many times.
Thus the conclusion to our problem seems quite unclear.\\

\noindent\textbf{Final remarks}\\

We may compare to the theory of minimal surfaces in $\mathbb{R}^4$.
In that case, we still have a pair of Gauss maps $\phi,\psi$ into
$\mathbb{C}P^1\times\mathbb{C}P^1$. This target space is endowed with
its standard K\"ahler form, and the unitary group action induced from
$\mathbb{C}P^3$. So it suffices to study the K\"ahler geometry of
$\mathbb{C}P^1\times\mathbb{C}P^1$. In particular,
Osserman's theorem \cite{Osser1964} and Fujimoto's theory \cite{Fujimoto}
are based on this observation.

For stationary surfaces in $\mathbb{R}^4_1$, the two Gauss maps
together gives a mapping
\begin{align*}
(\phi,\bar\psi):~M~~&\longrightarrow \quad Q^2\triangleq
\{(z,\bar{w})\in\mathbb{C}P^1\times\mathbb{C}P^1~|~z\ne \bar{w}\},\\
p~~&\mapsto \qquad (\phi(p),\overline{\psi(p)}).
\end{align*}
The target space is an open 2-dimensional complex manifold
homeomorphic to $S^2\times S^2$ with the diagonal removed.
The projective action of the Lorentz isometries on $\mathbb{R}^4_1$
induces an action of $SL(2,\mathbb{C})$ on this $Q^2$ (Remark~\ref{rem-trans}):
\[
A\cdot(z,\bar{w})~~\mapsto~~\left(\frac{az+b}{cz+d},~
\frac{a\bar{w}+b}{c\bar{w}+d}\right),~~~~~
A=\begin{pmatrix}a & b \\ c & d\end{pmatrix}\in SL(2,\mathbb{C}).
\]
Under this action we have an invariant complex 2-form
\begin{equation}
\Theta\triangleq\frac{1}{(z-\bar{w})^2}~dz\wedge d\bar{w},
\end{equation}
whose pull-back to $M$ via the mapping
$(\phi,\bar\psi)$ is exactly the curvature form
$-K+\mathrm{i}K^\bot$ up to a constant by \eqref{eq-totalcurvature}.
Thus we have to study a new geometry of
\[\left(Q^2,~\Theta,~SL(2,\mathbb{C})\right),\]
which is related with both
the problem of finite total curvature (Problem~\ref{prob-essential},
\ref{prob-collin}, \ref{prob-normalK})
and the value distribution problem of $\phi,\psi$
(Problem~\ref{prob-fujimoto}, \ref{prob-phipsi}).

Now let us explain several previous results from this viewpoint.
It is easy to see that there is no $SL(2,\mathbb{C})$-invariant
area form on any component of $\mathbb{C}P^1\times\mathbb{C}P^1$.
Thus Osserman's original argument (that finite total Gaussian curvature
implies no essential singularity of the Gauss map) does not apply at here.
In particular we have counter-examples in Section~3.

Singular points/ends now appear as intersections of the mapping
$(\phi,\bar\psi)$ with the diagonal
$\{(z,\bar{w})\in\mathbb{C}P^1\times\mathbb{C}P^1~|~z=\bar{w}\}$,
on which $\Theta$ tends to $\infty$. Thus at singular ends
the curvature integral is improper. It is a subtle question why
this integral converges absolutely exactly when
this is a good singular end. See Proposition~\ref{prop-goodsingular}.

\end{document}